\documentclass[letterpaper,12pt,reqno]{amsart}
\usepackage{amssymb}
\usepackage{cases}
\usepackage{amsmath}
\usepackage{mathrsfs}
\usepackage{eucal}
\usepackage{amsfonts}
\usepackage[all]{xy}
\usepackage{color}
\DeclareFontFamily{OT1}{pzc}{}
\DeclareFontShape{OT1}{pzc}{m}{it}{<-> [1.15] rpzcmi}{}
\DeclareMathAlphabet{\mathzc}{OT1}{pzc}{m}{it}

\usepackage[english]{babel}

\newtheorem{thm}{Theorem}[subsection]

\newtheorem{lem}[thm]{Lemma}
\newtheorem{cor}[thm]{Corollary}

\newtheorem{prop}[thm]{Proposition}

\theoremstyle{definition}

\newtheorem{defn}[thm]{Definition}

\newtheorem{rem}[thm]{Remark}
\newtheorem{rems}[thm]{Remarks}

\numberwithin{equation}{subsection}
\numberwithin{thm}{section}

\newcommand{\op}{{\text{\rm op}}}
\newcommand{\wH}{{\widetilde \sH}}

\newcommand{\rJ}{{\sqrt{J}}}

\newcommand{\im}{\mbox{Im}}

\newcommand{\sH}{{\mathcal H}}
\newcommand{\sT}{{\mathcal T}}

\newcommand{\sZ}{{\mathcal{Z}}}

\newcommand{\Ext}{{\text{\rm Ext}}}

\newcommand{\htt}{\operatorname{ht}}

\newcommand{\id}{\operatorname{Id}}

\newcommand{\Amod}{A\mbox{--mod}}

\newcommand{\Hom}{\text{\rm Hom}}

\newcommand{\End}{\operatorname{End}}

\newcommand{\sO}{{\mathscr{O}}}

\newcommand{\modH}{\mbox{mod--}H}

\newcommand{\wT}{{\widetilde{\sT}}}

\newcommand{\Hmod}{H{\text{\rm --mod}}}

\def\sT{{\mathcal T}}

\def\sQ{{\mathscr Q}}

\def\sO{{\mathcal O}}

\newcommand{\nsZ}{{\mathcal Z}^\natural}
\newcommand{\Spec}{{\text{\rm Spec}}}

\begin{document}
 
\title[Local and global methods in representations of Hecke algebras]{Local and global methods in representations
of Hecke algebras}
   \author{Jie Du, Brian J. Parshall, and Leonard L. Scott}
 \address{School of Mathematics and Statistics\\University of New South Wales\\ UNSW Sydney 2052
}
 \email{j.du@unsw.edu.au {\text{\rm (Du)}}}

\address{Department of Mathematics \\
University of Virginia\\
Charlottesville, VA 22903} \email{bjp8w@virginia.edu {\text{\rm
(Parshall)}}}

\address{Department of Mathematics \\
University of Virginia\\
Charlottesville, VA 22903} \email{lls2l@virginia.edu {\text{\rm
(Scott)}}}

\thanks{This work was supported by a 2017 University of New South Wales Science Goldstar Grant (Jie Du), and the Simons Foundation (\#359360, Brian Parshall; \#359363, Leonard Scott).}
\subjclass{20C08, 20C33, 16S50, 16S80}
\keywords{quasi-hereditary algebra, stratified algebra, Hecke algebra, Schur algebra, left cell, endomorphism algebra, exact category, height one prime}

\begin{abstract} This paper aims  at developing a ``local--global"  approach for various types of finite
dimensional algebras, especially those related to Hecke algebras.  The eventual intention is to apply the methods 
and applications developed here to the cross-characteristic
representation theory of finite groups of Lie type.  The authors first review the notions of quasi-hereditary and stratified algebras over a Noetherian commutative ring. They prove that many global properties of these
algebras hold if and only if they hold locally at every prime ideal. When the commutative ring is sufficiently good, it is
often sufficient to check just the prime ideals of height at most one. These methods are applied to construct certain
generalized $q$-Schur algebras, proving they are often quasi-hereditary (the ``good" prime case) but always stratified. Finally, these results are used to prove  a
triangular decomposition matrix theorem for the modular representations of Hecke algebras at good primes. In
the bad prime case, the generalized $q$-Schur algebras are at least stratified, and a block triangular analogue
of the good prime case is proved, where the blocks correspond to Kazhdan-Lusztig cells. \end{abstract}

\maketitle
 
\tableofcontents
\section{Introduction.} The focus of this paper is the representation theory of generic Hecke algebras which
arise in the study of principal and related series of finite groups of Lie type in cross-characteristic. Actually, our
preference is for the broader context of algebras generalizing the Dipper-James notion of a
$q$-Schur algebra, on which we will (re)focus in a later paper. Nevertheless, the $q$-Schur algebra generalizations
we found recently in \cite{DPS15}, \cite{DPS17a} already have consequences for Hecke algebra theory, as the later sections of this paper show.

The $q$-Schur algebras of Dipper-James were originally used  to study representations
of $GL_n(q)$ in cross-characteristic. For some time,  these $q$-Schur algebras have been known to be quasi-hereditary, even over the ring
$\mathbb Z[t,t^{-1}]$ of integral Laurent polynomials (with $t^2=q$, an indeterminate). In the case of types besides $GL_n$,  
the use of quasi-hereditary algebras in cross-characteristic
theory, while a good starting  point, seems too restrictive, if one
is seeking a theory for all characteristics different from the defining characteristic $p$.
 
One approach involves weakening the notion of a quasi-hereditary
algebra $A$. In general, the definition of a quasi-hereditary
algebra depends on the existence of certain idempotent ideals
$J=AeA$ in $A$ and recursively defined factor algebras (replacing
A with A/J, and repeating). Here, $e^2=e$ and $eAe$ is required to
be a semisimple algebra. Then, the axioms imply the various module
categories $eAe$-mod form "strata" in $A$-mod, which collectively
``glue'' together to give all of $A$-mod, at a derived category
level. The notion of a standardly stratified algebra parallels that of a
quasi-hereditary algebra, but the condition that $eAe$ be semisimple
is not assumed. The categories $eAe$-mod still glue together to
give $A$-mod, as before. This provides a somewhat cruder picture of
$A$-mod, but one which is still quite useful. Stratified algebras
were first introduced in \cite{CPS96}, largely over fields, and then
a full study of their integral versions was begun in \cite{DPS98a}.
This later paper began a long-term project by the authors to apply
stratified algebras in cross-characteristic representation theory of finite
groups of Lie type. See \cite{DPS98a}, \cite{DPS98b}, \cite{DPS15},
\cite{DPS17a}, \cite{DPS17b}. See also \cite[\S 9]{CPS99} and \cite{BDK01}.

In particular, \cite{DPS98a} formulated a conjecture
providing the cross-characteristic representation theory of finite groups $G(q)$ of Lie
type with a kind of generalized $q$-Schur algebra $A^+$, defined
directly from the generic Hecke algebra $H$, of the same type as
$G$, over the ring $\sZ:=\mathbb Z[t,t^{-1}]$. The authors conjectured in \cite{DPS98a} that $A^+$ could be constructed
as an (integral) standardly stratified algebra, with strata described in terms
of Kazhdan-Lusztig cell theory. The conjecture was verified in that
paper for all rank 2 cases (some of which led to standardly stratified algebras
which were not quasi-hereditary). In \cite{DPS15}, the conjecture,
in a slightly modified form, was established, if $\sZ$ is replaced by
its localization at the prime ideal generated by a cyclotomic
polynomial $\Phi_{2e}(t)$ with $e\neq 2$. It was required also, to
be able to quote work of \cite{GGOR03}, to only treat the so-called
``equal parameter'' case. (The original conjecture itself, as well as
the modified version, allows unequal parameters. It just requires
that they appear in a Hecke algebra arising from the 
 BN-pair structure of a finite group of Lie type.)

In \S 2 of this paper we develop a theory for integral
quasi-hereditary algebras and height 1 prime ideals (in their base
rings) strong enough to deduce global results from such local
results as the above. It can also be used to 
deduce local results at maximal ideals from results at height 1 primes.  This is applied in
\S 5 to prove a  local ``triangular decomposition matrix'' theorem in
the spirit of Geck and Jacon \cite[Thm. 4.4.1]{GJ11} as well as a global version, giving, in particular, a new
way to think about their use of Lusztig's $a$-function. We use more general height functions
on quasi-posets arising in quasi-hereditary/stratified algebra 
theory.  See Theorem \ref{Theorem4.2} and Remark \ref{5rem}(b). Though this result and most of \S5 focus on quasi-hereditary algebras 
and ``good prime'' results, standardly stratified algebra methods from
\cite{DPS17a} and \cite{DPS17b} are useful.

The two latter papers  deal with $\sZ$-versions of $A^+$,  and also
with standardly stratified algebras (rather than just quasi-hereditary
algebras). In \S 6 we stick with standardly stratified algebras and
address the question of what can be said regarding  (analogs of) the triangularization theorem of Geck (see \cite[Thm
4.4.1]{GJ11}) for ``bad'' primes. We show that a similar formulation
can be obtained if the role of single ordinary irreducible
characters is suitably replaced by the characters of two sided Kazhdan-Lusztig cells.

Next, \S 3, omitted in the above discussion so far, gives a
candidate parallel treatment of the results in  \S 2, but in a
stratified algebra context.  Proposition \ref{proposition1}
may also be regarded as a useful part of this theory. These
results are largely not needed in the later sections \S\S\S 4,5,6. Also,
\S4 provides some technical results on Morita theory needed
later in \S5.

We mention that  Corollary 2.6 corrects (and provides
a generalization of) an old local-global result
\cite[Thm. 3.3(c)]{CPS90}.  
The use of height one primes in 
the statement and ``proof'' of  \cite[op.cit.]{CPS90}
was one of the inspirations for  our approach here.

Finally, we make a comment on terminology. A local commutative ring is, of course,
any commutative ring with a unique maximal ideal.  The main base rings used in
\cite{DPS15} are primarily local. The terminology for global rings is much less standard. 
Typical examples are $\mathbb Z$, $\mathbb Z[t]$, and also $\sZ:=\mathbb Z[t,t^{-1}]$, the latter used as the
main base ring in \cite{DPS17a}. In the present paper, we focus on $\sZ$ as well as the fraction
rings $\mathbb S^{-1}\sZ$, where $\mathbb S$ is a multiplicative monoid generated by an explicit (and small) finite set of nonzero
elements in $\sZ$. Here we regard these rings as global, and try to understand algebras over them in terms of localizations
$\sZ_{\mathfrak p}=(\mathbb S^{-1}\sZ)_\mathfrak p$ at prime ideals $\mathfrak p$ containing no element of $\mathbb S$, and
especially those $\mathfrak p$ of height $\leq 1$.

 \section{Localization of integral quasi-hereditary algebras (QHAs).}
Let $k$ be a commutative Noetherian ring (with 1). All algebras over $k$ are assumed to  be finitely generated as $k$-modules
(i.e., they are $k$-finite).
  For ${\mathfrak p}\in\Spec\,k$ and a $k$-module $X$, let
$$X({\mathfrak p}):=X_{\mathfrak p}/{\mathfrak p}X_{\mathfrak p}$$
be the resulting module for the residue field $k({\mathfrak p})$. The functor $X\mapsto X(\mathfrak p)=X\otimes_kk(\mathfrak p)$ from the category of $k$-modules to the category of $k(\mathfrak p)$-vector spaces is right exact.
If $X$ is a $k$-submodule of a $k$-algebra $A$, let $\overline {X({\mathfrak p})}$ be the image of the natural $k({\mathfrak p})$-map
$X({\mathfrak p})\longrightarrow A({\mathfrak p})$. In general, $A(\mathfrak p)$ is a finite dimensional $k(\mathfrak p)$-algebra, and, if $X$ is a (left, right, 2-sided) ideal in $A$, then $\overline{X(\mathfrak p)}$ is a (left, right, or 2-sided, respectively) ideal in $A(\mathfrak p)$.

\begin{defn}\label{defn1} Assume that the $k$-algebra $A$ is projective over $k$. An ideal $J$ in $A$ is called a {\it heredity ideal} provided the following conditions hold.

\begin{itemize}
\item[(0)] $A/J$ is projective over $k$.
\item[(1)] $J$ is a projective as a left $A$-module.
\item[(2)] $J^2=J$.
\item[(3)] The $k$-algebra $E:=\End_A({_AJ})$ is $k$-semisimple.
\end{itemize}
\end{defn}

The heredity ideal $J$ is of separable (resp., semisplit, split) type provided that $E$ is separable (resp., semisplit, split)
over $k$. Recall that a $k$-algebra $E$ is separable if the $(E,E)$-bimodule map $E\otimes_kE\to E$
is split. One says that $E$ is semisplit if it is a finite direct product of algebras, each of which is separable and has
center $k$ (i.e., each factor is an Azumaya algebra over $k$). If each factor is the endomorphism algebra of a finite projective
$k$-module, then $E$ is called split.\footnote{In particular, full matrix algebras over $k$ are split. This often occurs 
for quasi-hereditary algebras using integral standard modules, see \cite[Lem. 1.6]{DPS98b} and its proof.} 

 Semisimple algebras over commutative rings arise in the context of relative homological algebra.
Alternatively,  a (finitely generated) $k$-algebra
$E$ is $k$-semisimple if and only if the $k(\mathfrak p)$-algebras $E(\mathfrak p)$ are $k$-semisimple for all
$\mathfrak p\in \Spec\, k$. (See \cite[Thm. 2.1]{CPS90}. This implies, in particular, that any $k$-algebra
Morita equivalent to $E$ is $k$-semisimple. We will also need the facts that every separable $k$-algebra is $k$-semisimple, and that every  module projective over $k$ for a $k$-semisimple algebra $E$ is projective over $E$.  See \cite[pp. 133--135]{CPS90}.)

Most idempotent ideals $J$ dealt with in this paper have the form $J=AeA$ for 
an idempotent $e\in A$. Indeed, if the idempotent ideal
$J$ is projective as a left $A$-module, then
$A$ is Morita equivalent to an algebra $A'$ having the property that  the corresponding idempotent ideal $J'$ is
 generated by an idempotent $e'\in A'$. In fact, for some positive integer $n$, we can take $A'=M_n(A)$, so that
 $J'=M_n(J)=M_n(A)e'M_n(A)$ for some idempotent $e'\in M_n(A)$. (See \cite[Rem. 1.4(b)]{CPS90} for more details.)

The following proposition is new in an integral context. Note that properties (0), (3) in the definition of a heredity
ideal are not used.

\begin{prop}\label{proposition1} Suppose $A$ is an $k$-algebra which is projective over $k$,  and $J$ is an ideal in $A$ satisfying conditions
(1) and (2) of Definition \ref{defn1}, that is, $J=J^2$ is projective as a left $A$-module. Then the following statements
hold.

(a) $E:=\End_A({_AJ})$ is projective over $k$.

(b) If $J=AeA$ for some idempotent $e\in A$, then $E$ is Morita equivalent to $eAe$.

(c) If $J=AeA$ for some idempotent $e\in A$, then $eA$ is a projective left $eAe$-module. Also, the natural multiplication map $Ae\otimes_{eAe}eA\longrightarrow J$ is
bijective, even at the derived category level, i.e., $Ae\otimes^{\mathbb L}_{eAe}eA\overset\sim\longrightarrow J$.
\end{prop}

\begin{proof} (a) Actually, this holds provided $J$ is any  left ideal in $A$ which is projective as a left $A$-module. In fact, in that case
there is a left $A$-module $X$ so that $Y:=J\oplus X\cong ({_AA})^{\oplus n}$ is a free $A$-module. Then
$\End_A({_AJ})$ is a direct summand of $\End_A(Y)$. The latter identifies with $n\times n$ matrices over
$A^{\op}$ and hence is projective over $k$. This proves (a).

To prove (b), note that for each $x\in A$, the $A$-submodule $Aex$ of $_AJ$ is a homomorphic image of $Ae$ (via
the identity map $Ae\longrightarrow Ae$ composed with right multiplication by $x$). Clearly, $AeA$ is a sum of finitely
many such submodules $Aex$, $x\in A$, since $k$ and $A$ are left Noetherian. Thus, $J$ is a homomorphic
image of the left $A$-module $(Ae)^{\oplus n}$ for some positive integer $n$. Thus, since $_AJ$ is projective,
it is a direct summand of $(Ae)^{\oplus n}$. Also, $Ae$ is a direct summand of $J$, viz., $J=Ae\oplus J(1-e)$.
Of course, the module $N:= J(1-e)$ is also a homomorphic image and direct summand of $(Ae)^{\oplus n}$.

Thus, letting $m=n+1$, $Ae=M$ and $N=J(1-e)$, we fulfill the hypothesis of the following lemma.

\begin{lem}
 \label{badlemma} Let $M,N$ be finitely generated (left) modules for a $k$-algebra $A$, and let $m$ be a positive integer. Suppose
there is a split $A$-module epimorphism $\pi:M^{\oplus (m-1)}\twoheadrightarrow N$. Then $\End_A(M)$ and
$\End_A(M^{\oplus m})$ are each Morita equivalent to $\End_A(M\oplus N)$.  \end{lem}

\begin{proof} 
 Of course, $\End_A(M)$ and $\End_A(M^{\oplus m})$ are trivially Morita equivalent,
so it suffices to show that $\End_A(M^{\oplus m})$ and $\End_A(M\oplus N)$ are Moirta equivalent. We will use the fact that if
$f$ is an idempotent in an algebra $B$, then $B$ is Morita equivalent to the algebra $fBf$ whenever
$BfB=B$ (i.e., $Bf$ is a progenerator of $B$-mod).
In our case, we let $B:=\End_A(M^{\oplus m})$.

  Let
$f'\in \Hom_A(M^{\oplus m}, M\oplus N)\cong\Hom_A(M, M)\oplus\Hom_A(M^{\oplus (m-1)},N)$ be given by
 $f':=\id_M\oplus \pi$.  Here we have written $M^{\oplus m}=M\oplus M^{\oplus(m-1)}$. This distinguishes a
 ``first summand'' $M=M_1$. Then
 put $f=(\id_M\oplus\sigma)\circ f'$, where $\sigma$ is a fixed splitting of the projection $\pi:M^{\oplus (m-1)}\longrightarrow
 N$. It is clear that $f$ is an idempotent in $B$.

\smallskip\noindent
{\bf Claim}: $BfB=B$.

Write $M^{\oplus m}=M_1\oplus\cdots\oplus M_m$, where each $M_i=M$ concentrated in position $i$. Let $b\in B$.
To show $b\in BfB$, it suffices to consider the case where $b|_{M_i}=0$ for all but one index  $i$, call it $j$.
 Without loss, $j=1$.    (If $bu\in BfB$, where $u$ is a unit in $B$, then $b\in BfB$. Choose $u$ to be a unit interchanging
 $M_j$ and $M_1$, and fixing the other summands $M_i.$)
 Thus, $b=b\pi_1$, where $\pi_1:M^{\oplus m}
\to M^{\oplus m}$ is defined by $(x_1,\cdots , x_m)\mapsto (x_1,0,\cdots, 0)$. Since $\pi_1=f\pi_1$,  we have
that $b=b\pi_1=bf\pi_1
\in BfB$, proving the Claim, and then the lemma. \end{proof}

Part (b) of the Proposition follows.

Finally, we prove (c). There is an evident surjection $Ae\otimes_keA\twoheadrightarrow J$ of left
$A$-modules. Since $_AJ$ is projective, $_AJ$ is a direct summand of $Ae\otimes_keA$. Thus,
$eJ$ is a direct summand of $eAe\otimes_keA$ in $eAe$-mod.  Thus, $eJ$ is a projective $eAe$-module.
Next,
$$eA\subseteq (eAe)A\subseteq eJ\subseteq eA,$$
so $eJ=eA$. Thus, $eA$ is a projective left $eAe$-module.

In particular, $Ae\otimes_{eAe}eA\cong A\otimes^{\mathbb L}_{eAe}eA$, since $-\otimes_{eAe}eA$ is
  exact as, say,
a functor from $A\otimes_k(eAe)^\op$--mod  to $(A\otimes_kA^\op)$--mod. To complete the proof of (c), it remains to
show that the multiplication map $Ae\otimes_{eAe}eA\overset\mu\longrightarrow J$ is bijective. It is clearly
surjective, hence split as a map of left $A$-modules. Let $N$ be the kernel of $\mu$, and let $J'\subseteq Ae\otimes_{eAe}eA$
be a left $A$-submodule complementary to $N$ (a section of $\mu$). Then $e(Ae\otimes_{eAe}eA)\cong
e((1-e)Ae\otimes_{eAe}eA)\oplus eAe\otimes_{eAe} eA\cong eAe\otimes_{eAe}eA\cong eA.$ However,
$eJ'\cong eJ$ through the bijection $\mu|_{J'}$, while $eJ=eA$, as shown above. Hence, $eN=0$. In more detail:
$Ae\otimes_{eAe}eA=J'\oplus N$, so $e(Ae\otimes_{eAe}Ae)= eJ'\oplus eN\cong eA\oplus eN\cong e(Ae\otimes_{eAe}eA)
\oplus eN$ as $k$-modules.

Thus, the set $e(Ae\otimes_{eAe}eA)$, which clearly generates the left $A$-module $Ae\otimes_{eAe}eA$,
is contained in $J'$. Thus, $J'$ is all of $Ae\otimes_{eAe}eA$, forcing $N=0$. This finishes the proof of (c).
\end{proof}

 \begin{rems}\label{rem2.3} (a) The argument for part (c) of the above proposition was inspired by a parallel argument in the field case given in \cite[\S2.1]{CPS96}.

 (b) It is well-known \cite{CPS88} that quasi-hereditary algebras are ``left quasi-hereditary'' if and only if they
 are ``right quasi-hereditary.'' Looking ahead to standardly stratified algebras in \S3 below, \cite[p.42]{CPS96} provides examples
 where this left-right symmetry fails for standardly stratified algebras. However, (c) can be used to show that, in the stratified case,
 $A^{\op}$ does satisfy the full embedding condition \eqref{coho} below.  (In \cite[p. 180]{DPS98a}, this full embedding
 condition was taken as [a weaker version] of the notion of a standardly stratified algebras, while the notion used here was called
 a standardly stratified algebra.)
\end{rems}

 In the following result, let $k$ be a Noetherian integrally closed domain.  Let $A$ be a (finite) $k$-algebra which is
 projective over $k$. Let
 $K$ be the fraction field of $k$. If $J$ is an ideal in $A$, let $\sqrt{J}:=J_K\cap A$.
Also, recall that, if $\mathfrak p\in \Spec\,k$, then $\overline{J(\mathfrak p)}$ denotes the image of the natural map $J(\mathfrak p)\to A(\mathfrak p)$.

Now we have
\begin{lem}\label{thm1} Maintain the notation introduced above.
Let $J=AeA$ be an idempotent ideal of $A$ generated by an idempotent $e$. Assume that $A/\rJ$ is a projective $k$-module. Also,  for each ${\mathfrak p}\in\Spec\,k$ of height $\leq 1$,
assume that $\overline {J({\mathfrak p})}$ is a heredity ideal in $A({\mathfrak p})$ such that
$\End_{A({\mathfrak p})}({_{A({\mathfrak p})}}\overline{J({\mathfrak p})})$ is isomorphic to a direct product of central simple algebras over $k({\mathfrak p})$.
Then

(a) the map $J({\mathfrak p})\longrightarrow\overline{ J({\mathfrak p})}$ is bijective (and, thus,  an isomorphism of $(A({\mathfrak p}), A(\mathfrak p))$-bimodules);

(b) $J=\rJ$, and, furthermore, $J$ is a heredity ideal in $A$. In addition, $\End_A({_AJ})$ is semisplit (i.e., it is a direct product of
Azumaya algebras).
\end{lem}

\begin{proof} Of course, (a) is a consequence of (b), but we need to prove (a) first.

By the Peirce decomposition, $eAe$ is a direct summand of $A$ as a $k$-module. Thus, $eAe$ is itself
projective as a $k$-module. By hypothesis, $J(0)=\overline {J(0)}$ is a heredity ideal such that $\End_{A(0)}({_{A(0)}J(0))}$
is separable over $K$. Therefore, $eA(0)e=(eAe)(0)$ is also a separable algebra since the multiplication map $A(0)e\otimes_{eA(0)e}eA(0)\longrightarrow J(0)$ is an isomorphism.

Let $\mathfrak p$ be a prime ideal of $k$ having height 1. Because $\overline {J({\mathfrak p})}$ is a heredity ideal, $\overline {J({\mathfrak p})}\cong
A({\mathfrak p})\otimes_{eA({\mathfrak p})e}eA({\mathfrak p})$. By the universal property of tensor products,
there is a natural surjective map $\overline{ J({\mathfrak p})}\to J({\mathfrak p})$. Thus, $\dim \overline {J{(\mathfrak p})}
\geq\dim J({\mathfrak p})$. By definition, $\overline{J(\mathfrak p)}$ is an image of $J(\mathfrak p)$. We conclude
the surjection $\overline{J(\mathfrak p)}\to J(\mathfrak p)$ is an isomorphism, as is the defining surjection
$J(\mathfrak p)\twoheadrightarrow\overline{J(\mathfrak p)}$. This proves (a).

It follows also that $A_{\mathfrak p}/J_{\mathfrak p}$ is torsion free over the discrete valuation ring (DVR) $k_{\mathfrak p}$ \cite[Thm. 4.25]{Reiner75} (or \cite[Thm. 7.13]{Bass68}) , so that
$\sqrt{J_{\mathfrak p}}=J_{\mathfrak p}$. Also, $(\sqrt{J}/J)_{\mathfrak p}\subseteq (A/J)_{\mathfrak p}$ which is torsion
free, so $(\sqrt{J}/J)_{\mathfrak p}=0$. Thus,
\begin{equation}\label{fact} (\sqrt{J})_{\mathfrak p}=J_{\mathfrak p}.\end{equation}

% Observe that the free $k_{\mathfrak p}$-module $eA_{\mathfrak p}e$ has rank equal to the dimension of $eA({\mathfrak %p})e$
%and to the dimension of $eA(0)e=(eAe)_K$.

Continuing with $\mathfrak p$ as above, the algebra $E':=
    \End_{A(\mathfrak p)}({_{A(\mathfrak p)}J}(\mathfrak p))$ is, by hypothesis, a direct product of central simple
algebras over $k(\mathfrak p)$.  Since $\overline{J(\mathfrak p)}$ is, also by hypothesis, a heredity ideal in
$A(\mathfrak p)$, $eA(\mathfrak p)e$ is, by Proposition \ref{proposition1}(b) above, applied over $k(\mathfrak p)$,  Morita equivalent to $E'$. Thus, $(eAe)(\mathfrak p) \cong eA(\mathfrak p)e$
is also a direct product of central simple algebras over $k(\mathfrak p)$. This statement holds for all prime ideals
$\mathfrak p$ of height $\leq 1$. By \cite[Prop. 2.3]{CPS90} (which uses the
separability results
\cite[Prop. 4.6]{AG60b}),
the algebra $eAe$ is a direct product of Azumaya algebras over $k$, and, in particular, it is separable over $k$.
 In particular, since $eA$ is a projective $k$-module, it is a projective
$eAe$-module.  It follows that $Ae\otimes_{eAe}eA\subseteq (Ae\otimes_{eAe}eA)_K$. However, we have
that a surjection $Ae\otimes_{eAe}eA\to AeA$. Composing with the inclusion $AeA\subseteq (AeA)_K
\cong (Ae\otimes_{eAe} eA)_K$, we obtain the previous inclusion.
It follows that $_AJ\cong {_AA}e\otimes_{eAe}eA$, so $J$ is projective as a left $A$-module.
Applying Proposition \ref{proposition1}(b) again, we obtain that $\End_A({_AJ})$ is semisplit over $k$.

To complete the proof, it suffices to show that $\rJ=J$. Observe that $J$ is projective over $k$, since it is a direct
summand of $Ae\otimes_{eAe}(eAe)^{\oplus r}$ for some $r$. Thus, $J=\bigcap_{\htt(\mathfrak p)=1}J_{\mathfrak p}.$
However, for each $\mathfrak p\in\Spec\, k$ of height 1,
$$J_{\mathfrak p}=(\sqrt {J})_{\mathfrak p}\supseteq\sqrt{J}\supseteq J$$
using (\ref{fact}).
Hence,
$J=\bigcap_{{\htt(\mathfrak p)})=1}J_{\mathfrak p}$ must equal $\sqrt{J}$. Here we use \cite[Thm. 4.25]{Reiner75} and
$k$-projectivity of $J$. (See also \cite[Prop. 7.13]{Bass68}.)
 \end{proof}

 \begin{thm}\label{BigTheorem} Assume that $A$ is a finite $k$-algebra which is projective over a commutative Noetherian
 ring $k$. Let $J$ be an idempotent ideal in $A$.

 (a) For any given $\mathfrak p\in\Spec\,k$, if $\overline{J(\mathfrak p)}$ is a heredity ideal in
 $A(\mathfrak p)$, then $J(\mathfrak p)\cong\overline{J(\mathfrak p)}$, so that $J(\mathfrak p)$ identifies with
 an ideal in $A(\mathfrak p)$.
 Moreover,  $J$ is a heredity ideal if and only if for each $\mathfrak p\in\Spec\, k$, $\overline{J(\mathfrak p)}$ is a heredity ideal
 in $A(\mathfrak p)$.

 (b) $J$ is a heredity ideal of separable type if and only if, for every maximal ideal $\mathfrak m$ in $k$, 
 $\overline{J(\mathfrak m)}$ is a heredity ideal in $A(\mathfrak m)$ of
 separable type.

 (c) Now assume that $k$ is an integrally closed domain and $A/J$ is projective over $k$. Then $J$ is a heredity ideal of semisplit 
 type if and only if $\overline{J(\mathfrak p)}$ is a semisplit  heredity ideal in $A(\mathfrak p)$ for every prime ideal $\mathfrak p$ of height $\leq 1$.
 
 (d) Assume that $k$ is a regular domain of dimension at most 2, and that $A/J$ is projective over $k$.
    Then $J$ is a heredity ideal of split 
 type if and only if $\overline{J(\mathfrak p)}$ is a heredity ideal of split type in $A(\mathfrak p)$ for every prime ideal $\mathfrak p$ of height $\leq 1$.
 
   \end{thm}

 \begin{proof} We begin with the following

 \smallskip\noindent{\bf Claim:} For any $\mathfrak p\in \Spec \,k$ such that $\overline{J(\mathfrak p)}$ is
 a heredity ideal in $A(\mathfrak p)$, we have:

 \begin{itemize}\item[(i)] $(A/J)_{\mathfrak p}$ is $k_{\mathfrak p}$-projective; and

 \item[(ii)] $J(\mathfrak p)\cong \overline{J(\mathfrak p)}$.

 \end{itemize}

We first prove (ii). For this, we may temporarily take $A$ to be $A_{\mathfrak p}$ and then even pass to its
completion $\widehat A_{\mathfrak p}$,
a faithfully flat base change from $A_{\mathfrak p}$. At this point, with $A$ local and complete,
 we can assume that $J=AeA$
 for an idempotent $e$; see the discussion in \cite[\S 1]{CPS90}. Note that $\overline {J(\mathfrak p)}$ remains a heredity ideal
 in $A(\mathfrak p)$.  By the well-known field case of Proposition \ref{proposition1},
 \begin{equation}\label{eqn2}\overline {J(\mathfrak p)}\cong A(\mathfrak p)e\otimes_{eA(\mathfrak p)e} eA(\mathfrak p)\cong
 (Ae\otimes_{eAe}eA)\otimes_k k(\mathfrak p).\end{equation}
 The natural multiplication map $Ae\otimes_{eAe}eA\to AeA$ gives a surjection
 \begin{equation}\label{eqn3}(Ae\otimes_{eAe}eA)\otimes_k k(\mathfrak p)\twoheadrightarrow
 J\otimes_kk(\mathfrak p)=J(\mathfrak p)\end{equation}
 So $J(\mathfrak p)$ is a homomorphic image of $\overline{J(\mathfrak p)}$
 by (\ref{eqn2}). Since $\overline{J(\mathfrak p)}$ is defined as the image of a natural surjection $J(\mathfrak p)\twoheadrightarrow \overline{J(\mathfrak p)}$, the latter must be an isomorphism, proving (ii). By \cite[Lem. 3.3.1]{CPS90} (a well-known
 commutative algebra fact \cite[Lem. 4]{M80}), assertion (i) also follows. This completes the proof of the Claim.

\medskip We now prove (a). First, assume that $J$ is a heredity ideal in $A$. By the Claim above, $(A/J)_{\mathfrak  p}$ is $k_{\mathfrak p}$-projective for every prime ideal $\mathfrak p$ in $k$. In particular, this holds for every maximal ideal
 $\mathfrak m$ in $k$, so that $A/J$ is a projective $k$-module. Thus,
 the sequence $0\to J\to A\to A/J\to 0$ of $k$-modules is $k$-split. So, for any prime ideal $\mathfrak p$,
 the sequence remains split upon applying the functor $-\otimes_kk(\mathfrak p)$. Thus, $\overline{J(\mathfrak p)}\cong J(\mathfrak p)$,
 and we may identify $J(\mathfrak p)$ with its image in $A(\mathfrak p)$, for any prime ideal $\mathfrak p$. Since
 $J$ is idempotent, so is each $J(\mathfrak p)$ idempotent, and also $J(\mathfrak p)$ is a projective left $A(\mathfrak p)$-module, since it is obtained from the projective left $A$-module $_AJ$ by base change from $k$ to $k(\mathfrak p)$.
 Finally, since $_AJ$ is projective, its $k$-semisimple endomorphism algebra $\End_A({_AJ})$ has $\End_{A(\mathfrak p)}({_{A(\mathfrak p)}J}(\mathfrak p))$
 as its $k(\mathfrak p)$--base
 change. The latter algebra is, therefore, $k(\mathfrak p)$-semisimple.

Conversely, assume $\overline{J(\mathfrak p)}$ is a heredity ideal for all prime ideals $\mathfrak p$. By the Claim, applied
just for all maximal ideals, it follows that $A/J$ is projective over $k$.   By \cite[Lem. 3.3.2]{CPS90}, $_AJ$ is $A$-projective. Finally, since $_AJ$ is $A$-projective, it base changes to $\End_{A(\mathfrak p)} ({_{A(\mathfrak p)}J}(\mathfrak p))$,
which is $k(\mathfrak p)$-projective by hypothesis. This statement holds for all $\mathfrak p\in\Spec\,k$. Thus, $\End_A({_AJ})$ is $k$-semisimple.  This completes the
proof of (a).

Part (b) is proved similarly, but using \cite[Thm. 4.7]{AG60b}.

Next, consider (c). First, assume $J$ is a heredity ideal of semisplit type.  We need to show that if $\mathfrak p\in\Spec\,k$ has height $\leq 1$, then $\overline{J(\mathfrak p)}$
is a heredity ideal of semisplit type in the algebra $A(\mathfrak p)$.  By (a) above, we know that $\overline{J(\mathfrak p)}$
is a heredity ideal in $A(\mathfrak p)$ and $J(\mathfrak p)\cong\overline{J(\mathfrak p)}$. As noted in the proof of (a), the endomophism algebra $\End_A(J)$ base
changes to $\End_{A(\mathfrak p)}({_{A(\mathfrak p)}J}(\mathfrak p))$. Since $\End_A({_AJ})$ is a direct sum  (algebra-theoretic direct product) of
Azumaya algebras (i.e., it is semisplit), $\End_{A(\mathfrak p)}({_{A(\mathfrak p)}J}(\mathfrak p))$ is also a direct
sum of Azumaya algebras over $k(\mathfrak p)$; see \cite[Rem. 2.4]{CPS90}, \cite[Prop. 1.4, Cor. 1.6]{AG60b}.

Conversely, suppose that $\overline{J(\mathfrak p)}$ is a heredity ideal in $A(\mathfrak p)$ of semisplit type, for every
prime ideal $\mathfrak p$ of height $\leq 1$. We wish to show $J$ is an heredity ideal of semisplit type in the
$k$-algebra $A$. The reader may check that this statement holds if and only if it is true for the ideal $J'=M_n(J)$
in the algebra $A'=M_n(A)$ for any particular positive integer $n$. As noted above Proposition \ref{proposition1}, we
can choose $n$, so that $J'$ is generated by an idempotent. By Lemma \ref{thm1}(b), $J'$ is a heredity ideal in $A'$ of
semisplit type.
 
 Now consider (d). 
 We require the following
 
 \begin{lem} \label{awfullemma} (a) Let $\mathtt D$ be a DVR with maximal ideal $\mathfrak m$, fraction field $\mathtt F$ and residue
 field $\mathtt f=\mathtt D/\mathfrak m$. Let $B$ be an projective algebra over $\mathtt D$ with the property that $B_{\mathtt F}$ and $B_{\mathtt f}$
 are full $n\times n$ matrix algebras over $\mathtt F$ and $\mathtt f$, respectively.  Then $B$ is an full $n\times n$ matrix
 algebra over $\mathtt D$.
 
 (b) Suppose that $E$ is an Azumaya algebra projective over a regular domain $k$ of dimension $\leq 2$. Suppose
 that, for any prime ideal $\mathfrak p\in\Spec\,k$ of height $\leq 1$, $E_{\mathfrak p}$ is split.  Then $E$ is split.\end{lem}
 
 \begin{proof}  (a) is an exercise using Nakayama's lemma. 
 
 (b) If $\mathfrak p=0$, put $K:=E_{\mathfrak p}$. Then we have $E_K\cong \End_K(V)$,
 where $V=K^n$ for some $n>0$. Using \cite[Prop. 1.1.1]{DPS98a}, $V$ has a $E$-stable full $k$-lattice $N$. 
 Thus, identifying $E$ with its image in $\End_k(N)$ we have
 $E\subseteq\End_k(N)$. For any prime ideal $\mathfrak p$  of height $\leq 1$ in $k$,
 $E_{\mathfrak p}$ is split, by hypothesis, so that $E_{\mathfrak p}\cong \End_{k_{\mathfrak p}}(P)$ for a projective $k_{\mathfrak p}$-module $P$ (which depends on $\mathfrak p$).   In particular, $E_{\mathfrak p}$ is a maximal
 order in $\End_K(P_K)$, as is well-known \cite[Thm. 8.7]{Reiner75}. Clearly, the $E_{\mathfrak p}$-modules $P_{\mathfrak p}$ and
 $N_{{\mathfrak p}}$ have the same rank, so are isomorphic as $E_{\mathfrak p}$-modules \cite[Thm. 18.7i]{Reiner75}.
 Thus, $E_{\mathfrak p}=\End_{k_{\mathfrak p}}(N_{\mathfrak p})$. Intersecting over all $\mathfrak p$ of height
 $\leq 1$, we get that $E\cong\End_k(P)$ is split as required; see \cite[(11.3)]{Reiner75} and the well-known (local version
 of) Auslander-Goldman's criterion for projectivity of modules over regular domains of dimension $\leq 2$ \cite[p. 18]{AG60a}. This proves the Lemma. \end{proof}
 
 Return to the proof of (d). First, suppose that $J(\mathfrak p)$ is a heredity ideal in $A(\mathfrak p)$ of split type, for
 every prime ideal $\mathfrak p$ of height $\leq 1$ of the regular domain $k$ of dimension $\leq 2$. We will
 show that $J$ is an heredity ideal of split type in the $k$-algebra $A$. (This will prove the $\Longleftarrow$ direction in (d). 
 We leave the $\Longrightarrow$ direction to the reader; it is not needed later in this paper.) The righthand hypothesis of
 (d) implies the righthand hypothesis of (c). The proof of (c) shows that $E=\End_A({_AJ})$ is a direct product of 
 Azumala algebras which may be taken as projections of $E$ onto the central simple components of $E_K$. For
 a prime ideal $\mathfrak p$ of height 1, $E({\mathfrak p})$ is a direct product of full matrix algebras. Clearly, if $B$
 is such a projection, the same
 statement holds if $E$ is replaced by $B$. Dimension considerations show that if $B(\mathfrak p)$ is itself a
 full matrix algebra. By (a) of the above lemma, $B_{\mathfrak p}$ is a full matrix algebra. Now by part (a) of
 the lemma, the $E$ is a full matrix algebra over $k$.
 \end{proof}

 Recall that the projective $k$-algebra $A$ is called a quasi-hereditary algebra (QHA) provided there exists a finite ``defining
 sequence'' $0=J_0\subseteq J_1\subseteq J_2\subseteq\cdots\subseteq J_t=A$ of ideals in $A$ such that for $0<i\leq t$,
 $J_i/J_{i-1}$ is a hereditary ideal in $A/J_{i-1}$.  In case  $k$ is a field, this definition agrees with the classical notion
 of a QHA given in \cite{CPS88}.  Given such a defining sequence $\{J_\bullet\}$, we say that it is
 a {\it defining sequence of separable type} provided that each $J_i/J_{i-1}$ is a heredity ideal of separable type, i.e.,
 $\End_{A/J_{i-1}}({_{A/J_{i-1}}J}_i/J_{i-1})$ is of separable type, $i=1,\cdots, t$.  A defining sequence of semisplit type, etc., can
 be defined similarly.

We end this section with the following improvement (and correction---see the remark following it) of \cite[Thm. 3.3]{CPS90}.
The proof is easily obtained from Theorem \ref{BigTheorem}, and further details are omitted.
%Corollary \ref{OldCPSTheorem}
 \begin{cor}\label{OldCPSTheorem} Let $A$ be as in Theorem \ref{BigTheorem}. Assume that
 \begin{equation}\label{sequence} 0=J_0\subseteq J_1\subseteq J_2\subseteq\cdots \subseteq J_t=A\end{equation}
 is a sequence of idempotent ideals. The following statements
 hold.

 (a) The $k$-algebra $A$ is quasi-hereditary with defining sequence (\ref{sequence}) if and only if, for each $\mathfrak p\in
 \Spec\,k$, the algebra $A(\mathfrak p)$ is quasi-hereditary with defining sequence
 \begin{equation}\label{sequence'}
 0=\overline{J_0(\mathfrak p)}\subseteq \overline{J_1(\mathfrak p)}\subseteq\overline{J_2(\mathfrak p)}\subseteq
 \cdots \subseteq\overline{J_t(\mathfrak p)}= A(\mathfrak p).\end{equation}
 (For any given $\mathfrak p\in\Spec\,k$, when these conditions
hold, the isomorphisms $\overline{J_i(\mathfrak p)}\cong J_i(\mathfrak p)$ are also valid.)

 (b) The k-algebra $A$ is quasi-hereditary of separable type with separable type defining sequence (\ref{sequence}) if and
 only if, for each maximal ideal $\mathfrak p$ in $ k$,  the $k(\mathfrak p)$-algebra $A(\mathfrak p)$ is quasi-hereditary of separable type with separable type defining sequence (\ref{sequence'}).

(c) Assume that $k$ is a Noetherian integrally closed domain. Assume also that for each $i$, $0\leq i<t$, $A/J_i$ is projective over $k$. The $k$-algebra $A$ is then quasi-hereditary of  semisplit type with semisplit type defining sequence (\ref{sequence}) if and only if, for each prime ideal  $\mathfrak p\in\Spec\,k$ of height $\leq 1$, the $k(\mathfrak p)$-algebra $A(\mathfrak p)$ is quasi-hereditary of semisplit type with semisplit  type defining sequence
(\ref{sequence'}).

(d) Assume that $k$ is a regular domain of dimension $\leq 2$. Assume also that for each $i$, $0\leq i<t$, $A/J_i$ is projective over $k$. The $k$-algebra $A$ is then quasi-hereditary of  split type with split type defining sequence (\ref{sequence}) if and only if, for each prime ideal  $\mathfrak p\in\Spec\,k$ of height $\leq 1$, the $k(\mathfrak p)$-algebra $A(\mathfrak p)$ is quasi-hereditary of split type with split  type defining sequence (\ref{sequence'}). \end{cor}

\begin{rems}\label{rem2.7} (a) Parts (a) and (b) of Corollary \ref{OldCPSTheorem} are essentially the same as parts (a) and (b) in
\cite[Thm. 3.3]{CPS90}.  We essentially adapted the arguments given in \cite{CPS90} to obtain
Theorem \ref{BigTheorem}(a), (b)
above.
 Part (c) of \cite[Thm. 3.3]{CPS90} is parallel to Corollary \ref{OldCPSTheorem}(c) above, but for
a smaller class of algebras $k$, omitting the assumption
 that the $A/J_i$ be projective over $k$. This omission appears
to be a mistake, and in any case the proof given in \cite{CPS90} is incorrect. For example, the assertion on \cite[p. 141]{CPS90}
that the $(E(\mathfrak m), A(\mathfrak m))$-bimodule $J(\mathfrak m)$ has $\overline{J(\mathfrak m)}$ as a homomorphic image appears to be wrong.

(b) Rouquier \cite[Thm. 4.15]{Ro08} gives a variation on Corollary \ref{OldCPSTheorem}(b) in a highest weight category setting. In the present
context, this is very close (but not identical)
to using heredity ideals $J_i/J_{i-1}$ of split type, and Corollary \ref{OldCPSTheorem}(c) holds as written using defining ideals
with this property with a proof similar to the proof above in the semisplit case. A similiar remark holds for part (b) in
Theorem \ref{BigTheorem}.
 
\end{rems}

\section{Stratified algebras and their localizations.} This section follows the outline of the previous section on integral
QHAs. The idea is to weaken the notion of a heredity ideal. As we see elsewhere, the
new class of algebras, called {\it standardly stratified algebras} (or SSAs), arise naturally in the study of the
cross-characteristic representation theory of finite groups of Lie type. Stratified algebras over a field (with some discussion over
DVRs) were first
introduced in \cite{CPS96}. The version we follow here, valid over general commutative rings, was first given in \cite{DPS98a}.
Essentially, condition (3) in the definition of a heredity ideal in Definition \ref{defn1} is dropped to give the notion of a {\it standard
stratifying} ideal. In particular, Proposition \ref{proposition1}(c) of the previous section could have been used to begin this section.

As in \S2, let $k$ be a Noetherian commutative ring, and let $A$ be a $k$-algebra, always assumed to be a finite
$k$-module which projective over $k$. We make the following definition, analogous to the notion of a heredity ideal.

\begin{defn}\label{SSI}  An ideal $J$ in $A$ is called a {\it standard stratifying ideal} if the following conditions hold.

\smallskip\noindent
(0) $A/J$ is projective over $k$;

\smallskip\noindent
(1)  $J$ is $A$-projective as a left $A$-module.

\smallskip\noindent
(2) $J^2=J$.
\end{defn}
\medskip

Observe that, in particular, a heredity ideal is a standard stratifying ideal. With the above notion, we can make the following
definition.

\begin{defn}\label{SSA} The algebra $A$ over $k$ is called a {\it standardly stratified algebra} (SSA) provided there exists
a finite ``defining sequence'' $0=J_0\subseteq J_1\subseteq J_2\subseteq\cdots\subseteq J_t=A$ of ideals in $A$
such that, for $0<i\leq t$, $J_i/J_{i-1}$ is a standard stratifying ideal in $A/J_{i-1}$.
\end{defn}

Thus, if $A$ is a QHA over $k$, it is also a standardly stratified algebra over $k$. The definition of both types of
algebras are given by defining sequences, but the requirements on the sections $J_i/J_{i-1}$ are weaker in the SSA case.

\begin{rems}\label{rem3.6}
(a) Suppose that $J$ is an ideal in a $k$-algebra $A$. Let $M,N$ be $A/J$-modules. Of course,
they can by inflation be regarded as $A$-modules. Thus, for any integer $n\geq 0$, there is
a $k$-module homomorphism
\begin{equation}\label{coho}\Ext^n_{A/J}(M,N)\longrightarrow\Ext^n_A(M,N).\end{equation}
By \cite[Appendix B]{DPS17a}, if $J$ satisfies conditions (1), (2) of Definition \ref{SSI}, then the maps
in (\ref{coho}) are isomorphisms for all $n\geq 0$. In particular, they are isomorphisms provided that $J$ is
a standard stratifying ideal or a heredity ideal in $A$.

(b) Stratified (and quasi-hereditary) algebras often arise naturally as endomorphism algebras $\End_R(T)$,
where $T$ is an appropriate (right) module for a $k$-algebra $R$. See \cite[(1.2.9) and Thm. 1.2.10]{DPS98a},
as well as the discussion in \S\S4,5 below.
\end{rems}

Let $A$ be a finite $k$-algebra which is projective over $k$. Let $J$ be an ideal in $A$, and assume (for simplicity)
that $A/J$ is projective over $k$.  The following proposition gives a simple condition that guarantees that $J$ 
is a standard stratifying ideal of $A$. 

\begin{prop}\label{prop1}  Let $A$ and $J$ be as immediately above. Suppose there is an idempotent $e\in J$, a $k$-subalgebra $E$ of $eAe$, and a projective $E$-submodule
$P$ of the (left) $E$-module $eA$ such that the natural multiplication map
\begin{equation}\label{mult}Ae\otimes_E P\overset\mu\longrightarrow J\end{equation}
is bijective.  Then $J^2=J$ and $_AJ$ is projective.\end{prop}

\begin{proof} Obviously the image of the multiplication map $\mu$ in (\ref{mult}) is contained in $AeP\subseteq AeA\subseteq J$. So, if
$J=\im\,\mu$, then $J=AeA$ is idempotent. (This part of the proof only requires the surjectivity of  $\mu$.)

It remains to prove that $_AJ$ is projective. If the  multiplication map $\mu$ in (\ref{mult}) is bijective, it gives an isomorphism of left $A$-modules.  Thus, $J$ is isomorphic to a direct summand of a sufficiently large direct sum $(Ae)^{\oplus n}$ of copies of $Ae$. This follows from the projectivity of $P$ as an
$E$-module, which implies that $P$ is a direct summand of $E^{\oplus n}$ for some positive integer $n$. \end{proof}

\begin{rem}\label{rem3.3} Proposition \ref{proposition1}(c) provides a converse to Proposition \ref{prop1} above. Specifically, suppose that
$J$ is a standard stratifying ideal in a $k$-algebra $A$ (still assumed to be a finite module which is projective over $k$). Then
by Definition \ref{SSI}, $J^2=J$ and $_AJ$ is a projective $A$-module.  Assume that $J=AeA$ for an idempotent
$e\in A$. By Proposition \ref{proposition1}(c), there is a surjective  map  (\ref{mult}) of $k$-modules taking $E=eAe$ and $P=eA$
(which is a projective left $E=eAe$-module). In fact, $\mu$ is an isomorphism.   \end{rem}

\begin{prop} \label{prop3.4} Assume that $k$ is an integrally closed Noetherian domain. Let $A$ and $J$ be as immediately above Proposition \ref{prop1}. Let $e\in J$ be an
idempotent and let $E$ be a $k$-subalgebra of $eAe$. Assume, for each $\mathfrak p\in\Spec\,k$ of height
$\leq 1$, that $E(\mathfrak p)$ is a direct product of copies of $k(\mathfrak p)$. 
 Assume $E$ is projective over $k$,
and that $P$ is an $E$-submodule, projective over $k$, of $eA$ such that multiplication induces an isomorphism
$$(Ae)(\mathfrak p)\otimes_{E(\mathfrak p)}P(\mathfrak p)\longrightarrow J(\mathfrak p)$$
for each $\mathfrak p\in Spec\,k$ of height $\leq1$.  Then the multiplication map (\ref{mult}) is an isomorphism, and 
 $P$ is projective over $E$. In particular, the conclusions of Proposition \ref{prop1} hold.\end{prop}

\begin{proof}Since $E$ is projective over $k$, and $E(\mathfrak p)$ is semisimple for each $\mathfrak p\in\Spec\,k$ having
height $\leq 1$, $E$ is the product of its projections onto simple factors of $E\otimes_kK$. (See \cite[Prop. 2.3b]{CPS90}, which uses also the fact that $k$ is an integrally closed domain with fraction field $K$.) It follows that $E$ is a direct product of copies of $k$.

In particular, $E$ is separable as a $k$-algebra in the classical sense of Auslander-Goldman \cite{AG60b}, 
and so each $E$-module projective over $k$
is projective as an $E$-module. (See the discussion in \cite[p. 133]{CPS90}.) In particular,
$P$ is projective as an $E$-module.

It remains to show that the multiplication map $\mu$ in (\ref{mult}) is an isomorphism.
 By hypothesis, for $\mathfrak p=(0)$, $\mu$
becomes an isomorphism upon base change to $K$. However, the natural map $Ae\otimes_EP\longrightarrow
(Ae\otimes_EP)_K$ factors as the composite of the map $Ae\otimes_EP\to (Ae)_K\otimes_EP$ (which is an injection, using the projectivity of $P$ over
$E$) with the natural invertible map
$$(Ae)_K\otimes_EP=(K\otimes_kAe)\otimes_EP\longrightarrow K\otimes_k(Ae\otimes_EP)=(Ae\otimes_EP)_K.$$ Thus,
the  map $\mu$ in (\ref{mult}), when followed by the inclusion $J\subseteq J_K$, becomes an injection. (Note that $J$ is $k$-torsion
free.) Hence, the displayed map itself is an injection.

Note that the injectivity of $\mu$ gives an isomorphism of $Ae\otimes_EP$ with its image $AeP
\subseteq J$.

It remains to show that $\mu$ is surjective. By hypothesis, the
 isomorphism $(Ae)(\mathfrak p)\otimes_{E(\mathfrak p)}P(\mathfrak p)\overset\sim\longrightarrow J(\mathfrak p)$ gives a surjection
$$\begin{aligned}(Ae\otimes_EP)_{\mathfrak p }\cong(Ae)_{\mathfrak p}\otimes_EP \cong(Ae)_{\mathfrak p}\otimes_{E_\mathfrak p}E_{\mathfrak p}\otimes_EP  &\\ \cong (Ae)_{\mathfrak p}\otimes_{E_{\mathfrak p}}P_{\mathfrak p}
\twoheadrightarrow
(Ae)(\mathfrak p)\otimes_{E(\mathfrak p)}P_{\mathfrak p}\cong J(\mathfrak p),\end{aligned}$$
compatible with multiplication, for each height one prime ideal $\mathfrak p$. So multiplication gives surjections $(Ae\otimes_EP)_{\mathfrak p}\twoheadrightarrow J_{\mathfrak p},$ since $J(\mathfrak p)$ is the head of the finitely
generated $k$-module $J_{\mathfrak p}$. Since $P$ is a projective $E$-module, and $Ae$ is
 projective over $k$, the $A$-module $AeP\cong Ae\otimes_EP$ is projective over $k$. Hence,
$$ AeP=\bigcap_{\htt(\mathfrak p)=1}(AeP)_{\mathfrak p}$$
where $\htt(\mathfrak p)$ denotes the height of the prime ideal $\mathfrak p$.   The intersection is taken in $(AeP)_K$.
However, the inclusion $AeP\subseteq J$ induces an isomorphism $(AeP)_K\overset\sim\longrightarrow J_K$, which we use to identify $J$ with a submodule of $(AeP)_K$.  We may similarly regard any $J_{\mathfrak p}$ as a $k_{\mathfrak p}$-submodule of $(AeP)_K$ containing, and, in fact, equal to $(AeP)_{\mathfrak p}$. It is still true that
$J\subseteq\bigcap_{\htt(\mathfrak p)=1} J_{\mathfrak p}$. Thus, $J\subseteq AeP\subseteq J$ in $(AeP)_K$. The
resulting equality $AeP=J$ holds as well back in $J_K$. We have now obtained all the hypotheses of Proposition \ref{prop1},
and, so the proof of this proposition is complete.
\end{proof}

We conclude this section with a brief discussion of stratifying systems. These are analogous 
 to highest weight category structures for a module category $A$-mod.

By a quasi-poset, we mean a (usually finite) set $\Lambda$ with a transitive and reflexive relation $\leq$.  (In other words,
$\leq$ is a pre-order on $\Lambda$.) An
equivalence relation $\sim$ is defined on $\Lambda$ by putting $\lambda\sim\mu$ if and only if $\lambda\leq\mu$ and $\mu\leq\lambda$. Let $\bar\lambda$ be the equivalence class containing $\lambda\in\Lambda$. Of course, $\bar\Lambda$
inherits a poset structure. 

 Given a finite quasi-poset $\Lambda$, a {\it height function} on $\Lambda$ is a mapping $\htt:\Lambda\to\mathbb Z$ with the
properties that $\lambda<\mu\implies \htt(\lambda)<\htt(\mu)$ and $\bar\lambda=\bar\mu\implies\htt(\lambda)=\htt(\mu)$. 
We also say that the function $\htt$ is a height function compatible with quasi-poset structure. Given $\lambda\in\Lambda$,  a sequence
 $\lambda=\lambda_n>\lambda_{n-1}>\cdots>\lambda_0$ is called a chain of length $n$ starting at $\lambda=\lambda_n$.
 Then the standard height function $\htt:\Lambda\to\mathbb N$ is defined by setting $\htt(\lambda)$ to be the maximal
 length of a chain beginning at $\lambda$.

 We can now review the notion of a  {\it stratifying system}
for a finite
$k$-algebra $A$ and a quasi-poset $\Lambda$. We follow the discussion in \cite[\S2]{DPS15} fairly closely.
As noted there, in the original discussion of stratifying system \cite{DPS98a}, what we define below was called a ``strict'' stratifying
system. As in \cite{DPS17a}, we drop the word ``strict'' in our treatment.

%By a quasi-poset, we mean a (usually finite) set $\Lambda$ with a transitive and reflexive relation $\leq$. An
%equivalence relation $\sim$ is defined on $\Lambda$ by putting $\lambda\sim\mu$ if and only if $\lambda\leq\mu$ and $\mu\leq\lambda$. Let $\bar\lambda$ be the equivalence class containing $\lambda\in\Lambda$. Of course, $\bar\Lambda$ inherits a poset structure.

 \begin{defn}\label{strat} Let $k$ be as in the previous section, and let $A$ be a finite $k$-algebra which is 
 projective over $k$. Let $\Lambda$ be
  a quasi-poset.
  For $\lambda\in\Lambda$, there is given a finitely generated $A$-module $\Delta(\lambda)$, projective as a
 $k$-module\footnote{The assumption that each $\Delta(\lambda)$ is projective over $k$ was (incorrectly) omitted
 in \cite{DPS15}, though it was used in that paper; see also footnote 3 in \cite{DPS17a}.},
 and a finitely generated, projective $A$-module $P(\lambda)$, together with an epimorphism $P(\lambda)\twoheadrightarrow \Delta(\lambda)$.  The following
conditions are assumed to hold:

\begin{itemize}
\item[(1)] For $\lambda,\mu\in\Lambda$,
$$\Hom_{ A}(P(\lambda), \Delta(\mu))\not=0\implies \lambda\leq\mu.$$

\item[(2)] Every irreducible $A$-module $L$ is a homomorphic image of some $\Delta(\lambda)$.

\item[(3)] For $\lambda\in\Lambda$, the $A$-module $P(\lambda)$ has a finite filtration by $ A$-submodules
with top section $\Delta(\lambda)$ and other sections of the form $\Delta(\mu)$ with $\bar\mu>\bar\lambda$.
\end{itemize}
When these conditions all hold, the data consisting of the $\Delta(\lambda)$, $P(\lambda)$, etc. form (by definition) a  {\it  stratifying system} for the category
 $ \Amod$ of finitely generated $A$-modules.\end{defn}

  The modules $\Delta(\lambda)$, $\lambda\in\Lambda$, are called
 the {\it standard modules } for the stratifying system. It is also clear that $\Delta(\lambda)_{k'}, P(\lambda)_{k'}, \dots$ is a
 stratifying system for $ A_{k'}$-mod {\it for any base change $k\to k'$}, provided $k'$ is a Noetherian commutative
ring. (Notice that condition (2) is redundant, if it is known that the direct sum of the projective modules in (3) is a
progenerator---a property preserved by base change.)

We record the following useful result.

\begin{lem}(\cite[Lem. 2.1]{DPS17a} )  \label{firstlemma} Suppose that $A$ has a  stratifying system as above. Let $\lambda,\mu\in\Lambda$.
  Then
   $$\Ext^1_A(\Delta(\lambda),\Delta(\mu))\not=0
 \implies\lambda<\mu.$$ \end{lem}

 Given $A$-modules $X,Y,$ recall that the trace module ${\text{\rm trace}}_X(Y)$ of $Y$
 in $X$ is the submodule of $X$ generated by the images of all homomorphisms $Y\to X$.

 \begin{prop}\cite[Prop. 2.2]{DPS17a}     \label{prop2.2}  Suppose that $A$ has a  stratifying system as above, and let $\htt:\Lambda\to\mathbb Z$
 be a height function. Let $\lambda\in\Lambda$.  Then the $\Delta$-sections  arising from the
 filtration of  $P(\lambda)$ in Definition \ref{strat}(3)
 can be reordered (constructively, see its proof in {\it op. cit.}) so that, if we set
 $$ P(\lambda)_j=\text{\rm trace}_{P(\lambda)}\big(\bigoplus_{\htt(\mu)\geq j}P(\mu)\big),$$
 then $P(\lambda)_{j+1}\subseteq P(\lambda)_j$, for $j\in\mathbb Z$, and
 $$P(\lambda)_j/P(\lambda)_{j+1}$$
 is a direct sum of modules $\Delta(\mu)$ satisfying $\htt(\mu)=j$.
 \end{prop}
 
 \begin{proof} First, fix $j$ maximal with a section $\Delta(\mu)$ appearing in $P(\lambda)$ such that $\htt(\mu)=j$.  Lemma \ref{firstlemma} implies that, whenever $M$ is a module with a submodule $D\cong \Delta(\nu)$ and $M/D\cong \Delta(\mu)$, with $\mu,\nu\in\Lambda$ and $\htt(\nu)\leq \htt(\mu)$, then $M$ is the direct sum of $D$ and a submodule $E\cong \Delta(\mu)$. Of course the quotient $M/E$ is isomorphic to $D$. This interchange of $E$ with $D$ can be repeatedly applied to adjacent $\Delta$-sections in a filtration (SS3) of $P(\lambda)$ to construct a submodule $P(\lambda)(j)$, a term in a modified filtration, which is filtered by modules $\Delta(\nu)$ with $\htt(\nu)= j$, and $P(\lambda)/P(\lambda)(j)$ filtered by modules $\Delta(\nu)$ with $\htt(\nu)<j$.
 Axiom (SS1) clearly gives $P(\lambda)(j)=P(\lambda)_j$, and $P(\lambda)_{j+1}=0$. Clearly,
 $P(\lambda)_j/P(\lambda)_{j+1}$ is a direct sum as required  by the proposition. We have not used projectivity of $P(\lambda)$, only its filtration properties. Induction applied to the quotient module $P(\lambda)/P(\lambda)_j$ completes the proof.\end{proof}

 In \cite[Thm. 1.2.8]{DPS98a}, it is shown that if an algebra $A$ over $k$ has a stratifying system, then $A$ has a
 standard stratification involving idempotent ideals. For our purposes, the following result using a height function
 $\htt$ on $\Lambda$ is more useful.

 \begin{thm}  \label{Morita 3.10} Let $A$ be a finite projective $k$-algebra which has a stratifying system
 $$\{\Delta(\lambda), P(\lambda)\}_{\lambda\in\Lambda}.$$
 Put $P:=\bigoplus_{\lambda\in\Lambda}P(\lambda)$ and $A':=\End_A(P)^\op.$
 Then

 \begin{itemize}
 \item[(i)] $A'$ is Morita equivalent to $A$ by means of the functor
 $$\Hom_A(P,-):A\text{-mod}\longrightarrow A'\text{-mod}.$$

 \item[(ii)] the category $A'-$mod has a stratifying system $\{\Delta'(\lambda), P'(\lambda)\}_{\lambda\in\Lambda}$
 corresponding to $\{\Delta(\lambda), P(\lambda)\}_{\lambda\in\Lambda}$ under the Morita
 equivalence of (i).

 \item[(iii)] $A'$ is standardly stratified with a defining sequence $0=J'_0\subseteq J'_1\subseteq\cdots\subseteq J'_n=A'$,
 where $n=\max_{\lambda\in\Lambda}\{\htt(\lambda)\}$, and $J'_i/J'_{i-1}$ is a direct sum of modules $\Delta'(\lambda)$, each with
 $\htt(\lambda)=n+1-i$.
  \end{itemize}
  \end{thm}

  \begin{proof} The proof is an easy application of Proposition \ref{prop2.2}: Filter each $P(\lambda)$ and $P$ itself
  by standard modules $\Delta(\mu)$'s according to height as in Proposition \ref{prop2.2}. Note that there are
  no module homomorphisms $\oplus_{\htt(\mu)\geq j}P(\mu)\to P/P_j$ by axiom (1) and the (rearranged) version of axiom (3) in
  Definition \ref{strat}, where $P_j:=\oplus_{\lambda\in\Lambda}P(\lambda)_j$. It follows that
  $$\Hom_A(P,P_j)\subseteq\Hom_A(P,P)=A'$$
  is an ideal in $A'$, which we set to be $J'_{n-j+1}$. We also have the short exact sequence
  $$0\to\Hom_A(P,P_j)\longrightarrow\Hom_A(P,P)\longrightarrow\Hom_A(P,P/P_j)\to 0$$
  which identifies with the short exact sequence $0\to J'_{n-j+1}\to A'\to A'/J'_{n-j+1}\to 0$. Note that
  $P/P_j$ is projective over $k$ since it is filtered by various of the $\Delta(\mu)$. It follows that
  $$A'/J'_{n-j+1}\cong \Hom_A(P,P/P_j)$$
  is projective over $k$ (since $P$ is projective over $A$).

 The remaining details follow by induction and are left to the reader. Note that
 $$A/J_1\cong \Hom_A(P,P/P_n)\cong\Hom_A(P/P_n,P/P_n)
 \cong\Hom_{A/J_1}(P/P_n,P/P_n).$$
  \end{proof}

 \begin{rem}\label{rem3.11} (a) $A$ itself is also standardly stratified by a sequence of defining ideals
 $J_i$ corresponding to the sequence $J'_i$ under the Morita equivalence. However, we may have
 less control over the summands of various $J_i/J_{i-1}$. A remedy is to replace $A$ with $A'$.

 (b) Assume the above replacement has been made. There is another useful choice of $E$ and $P$
 in Proposition \ref{prop3.4} closer to Proposition \ref{prop1}:  Take $e=\sum e_\lambda\in A$. (Recall that $A$ is the now
 relabeled $A'$; here $e_\lambda$ is the projection $P\to P(\lambda)$ in the construction of $A'$.)  Let
 $E=\sum ke_\lambda$. For $P$ in Proposition \ref{prop3.4}, we will use an $E$-module $Q$ constructed as the
 direct sum of $E$-modules $Q_\lambda\subseteq e_\lambda A\subseteq eA$, $\lambda\in\Lambda$.
 Each $Q_\lambda$ is a free
 $k$-module spanned by elements $a_{\lambda,\mu,s}\in e_\lambda J_1$ 
 ($\cong \Hom_A(Ae_\lambda,J_1)$) where
 $\mu\in\Lambda$ and $s$ belongs to a set of integers indexed by the pair $\lambda,\mu$, such that 
  $J_1=\bigoplus_{\lambda,\mu,s} Ae_{\lambda,\mu,s} a_{\lambda,\mu,s}$ and
 $Ae_\lambda a_{\lambda,\mu,s}\cong\Delta(\lambda).$
  As a result, we get the hypothesis of Proposition \ref{prop3.4} with $E$ as above,
 $Q $ ($=P$ in Proposition \ref{prop3.4})
  a direct sum of $E$-modules isomorphic to various $ke_\lambda$.
  Thus, $Q$ is projective over $E$. A particular interest of this example
 is that the stratifying system can be reconstructed from this description.
\end{rem}
 
 \section{Some Morita equivalences.} Let $\mathcal Z={\mathbb Z}[t,t^{-1}]$ be the ring of
Laurent polynomials over the ring of integers $\mathbb Z$, and let $K$ be its fraction field. 
Let $\mathscr G=\{G(q)\}$ be a family of finite groups of Lie type, in the sense of
\cite[Section 68.22]{CR87}. The groups $G(q)$ each have a BN-pair structure and there is associated a
finite Coxeter system $(W,S)$ (which is independent of $q$).
For simplicity, we ignore the Ree and Suzuki groups in this paper. We will consider the generic Hecke algebra $H$ over $\mathcal Z$ with generators $T_s$, $s\in S$. It has
$\mathcal Z$-basis $\{T_w\}_{w\in W}$ and is defined by relations
\begin{equation}\label{relations}
T_sT_w=\begin{cases} T_{sw}, \quad{\text{\rm if}}\,\, sw>w;\\ 
t^{2c_s}T_{sw}+ (t^{2c_s}-1)T_w, \quad{\text{\rm if}}\,\, sw<w.\end{cases} 
\end{equation}
Recall the index parameters $c_s$, $s\in S$, are given by $[B(q): {^sB}(q)\cap B(q)]=q^{c_s}$, where $B(q)$ is a Borel subgroup. For any commutative
$\mathcal Z$-algebra $R$, let $H_{R}:=R\otimes_{\mathcal Z}H$.

The distinct irreducible left (or right) $H_K$-modules are indexed by a finite set, which we denote by $\Lambda$. For
$\lambda\in \Lambda$, let $E_K(\lambda)$ denote the associated irreducible left $H_K$-module.\footnote{Thus, $\Lambda$ also indexes the irreducible
$\mathbb QW$-modules. For $^2F_4$, which is not allowed here, the algebras $ \mathbb Q W$ and $H_K$ are not split.}

Recall the pre-order $\leq_{LR}$ on $W$ \cite[Ch. 8]{Lus03}. Its associated equivalence classes are called two-sided (Kazhdan-Lusztig) cells.   From its definition, $\leq_{LR}$ induces a pre-order, denoted again by $\leq_{LR}$, on the set $\Xi$ of two-sided
cells in $W$.

The $\mathcal Z$-spans of the Kazhdan--Lusztig basis elements over the two-sided 
cells  are the sections in a filtration of $H$ by two-sided ideals, and there is a unique decomposition of the (split) semisimple algebra
$H_K$ as a direct product $\mathfrak C_1\times\cdots\times\mathfrak C_r$ of (semisimple) two-sided ideals, one for each
two-sided cell section. This provides a corresponding decomposition of $\Lambda$. Namely, given $\lambda\in\Lambda$, let $[\lambda]\subseteq \Lambda$, be the set of those $\mu$
such that $E_K(\mu)$ and $E_K(\lambda)$ are modules for the same $\mathfrak C_i$.
Each $\lambda\in\Lambda$ determines a unique two sided cell, which we denote by $c[\lambda]$: thus, $E_K(\lambda)$ is a left $H_K$-module for a unique ideal $\mathfrak C_i$ determined by $c[\lambda]$. Then $\leq_{LR}$ gives rise to a corresponding pre-order on $\Lambda$, setting $\lambda\leq_{LR}\mu$ if and
only if $c[\lambda]\leq_{LR}c[\mu]$. Similarly, the opposite pre-order $\leq^{\op}_{LR}$ may be defined
on $\Lambda$. 

   Let
$\text{ht}:\Lambda\to\mathbb Z$ be a height function compatible with the quasi-poset structure defined by
$\leq_{LR}^{\op}$. (See the discussion above Definition \ref{strat}.) 
For convenience, we
can assume that $\htt$ has image in $\mathbb N$. 

There is also a pre-order $\leq_{L}$ on $W$ whose associated equivalence classes are called left cells. (It is finer
than the pre-order $\leq_{LR}$ on $W$, i. e., $x\leq_Ly\implies x\leq_{LR} y$, for $x,y\in W$.) Let $\Omega$ be the set of left cells \cite{Lus03}.  For each left cell $\omega\in\Omega$, there is a corresponding
 left cell module $S(\omega)\in \Hmod$; see \cite[below Rem. 4.8]{DPS17a} or \cite[\S8.3]{Lus03}. 
Observe, by definition, that $\htt$ takes a constant value on left cells occurring in the same two-sided cell. 
 The $(H,H)$-bimodule decompositon above, of $H_K$ into the direct sum of all two-sided cell modules, can be refined
 into   a (left) $H$-module decomposition of $H_K$ into the direct sum over $\omega\in\Omega$ of all left cell modules
 $S(\omega)$. Consequently, given
two left cells $\omega,\omega'$, if $S(\omega)_K$ and $S(\omega')_K$ have a common composition factor,
then $\omega$ and $\omega'$ are contained in the same two-sided cell, and so $\htt(\omega)=\htt(\omega')$.
In particular, the function $\htt$ is well-defined on the set of left cells. Also, $\leq^\op_{LR}$ makes sense on
$\Omega$, just as it does on $W$, $\Lambda$ and $\Xi$.  In addition, $\htt:\Omega\to\mathbb Z$ is compatible with $\leq^\op_{LR}$.

As in \cite[\S3]{DPS17a},
 we will use the ``dual left cell modules'' $S_\omega:= \Hom_{\mathcal Z}(S(\omega),\mathcal Z)\in\modH$. Now for an integer
 $i$, let $\mathscr S_i$ be the full additive subcategory of $\modH$ whose objects are finite direct sums of 
 various (repetitions allowed) dual left cell modules $S_{\omega}$ with $\htt(\omega)=i$.
This notation agrees with that in \cite[op. cit.]{DPS17a} (except that our $X_j$ below would be denoted $X^j$ 
there). The (full) additive category $\mathscr A(\mathscr S)$ 
defined there consists of objects $X$, with a (finite) filtration $\cdots \supseteq X_j\supseteq X_{j-1}\supseteq \cdots
$ by right $H$-modules $X_j$ satisfying
$X_j/X_{j-1}\in \mathscr S_j$, for each $j$.  Of course, these filtrations depend on the height function $\htt$.

Note that the smallest nonzero filtration term $X_i$ of $X$ has the property that $X_i\in\mathscr S_i$. In
\cite{DPS17a} we constructed finite dimensional right $H$-modules $X_\omega$, $\omega\in \Omega$.
See, in particular, the discussion immediately above \cite[Thm. 4.9]{DPS17a}, which uses \cite[Thm. 4.7]{DPS17a}.
There is an exact category $({\mathscr A}(\mathscr S), {\mathscr E}({\mathscr S}))$ constructed
in \cite[Construction 3.8]{DPS17a}. Here $\mathscr E(\mathscr S)$ is defined by all short exact sequences (in $H$-mod) of objects in $\mathscr A(\mathscr S)$ which remain exact, and are even split, when passing to a section defined by any $\mathscr S_j$.
 Now put
\begin{equation}\label{dagger}
T^\dagger=\bigoplus_{\omega\in\Omega}X_\omega^{\oplus m_\omega}\end{equation}
for any fixed set $\{m_\omega\}_{\omega\in\Omega}$ of positive integers. The following properties hold:

\medskip
(1) For all $\omega\in\Omega$,  $X_\omega\in\mathscr A(\mathscr S)$, and its smallest nonzero filtration (with respect to $\htt$) term is isomorphic to $S_\omega$;
 
 \smallskip
 (2) For all $\omega\in \Omega$, 
 $$\Ext^1_{\mathscr E(\mathscr S)}(S_\omega, T^\dagger)=0.$$ 
 We remark that if $X\in \mathscr A(\mathscr S)$, all exact sequences
 $X_j\hookrightarrow X\twoheadrightarrow X/X_j$ belong to $\mathscr E(\mathscr S)$, $j\in\mathbb N$.
 If $Y\in \mathscr A(\mathscr S)$ is such that $\Ext^1_{\mathscr E(\mathscr S)}(S,Y)=0$ for all $S\in\mathscr S$,
 then $\Hom_H(-,Y)$ applied to an exact sequence in $\mathscr E(\mathscr S)$ yields an exact sequence;
 see \cite[Lem. 3.10]{DPS17a}. This will be used in the proof of the following result.

\begin{thm}\label{newlemma} Let $X\in\mathscr A(\mathscr S)$ satisfy 
$\Ext^1_{\mathscr E(\mathscr S)}(S_\omega,X)=0$ for all $\omega\in\Omega$. Let $T^\dagger$ be as in (\ref{dagger}). Put $T^{\dagger\prime}= T^\dagger\oplus X$. Then the $\mathcal Z$-endomorphism algebras
$A^\dagger=\End_H(T^\dagger)$ and $A^{\dagger\prime}=\End_H(T^{\dagger\prime})$ are Morita equivalent.
A specific Morita equivalence 
$$A^{\dagger\prime}-{\text{\rm mod}}\overset\sim\longrightarrow A^\dagger-{\text{\rm mod}}$$
 is given on objects
by $N'\mapsto eN'$, where $e:T^{\dagger\prime}\to T^\dagger\subseteq T^{\dagger\prime}$ is projection from
$T^{\dagger\prime}$ to $T^\dagger$ along $X$,  and $eA^{\dagger\prime}e$ is identified with $A^\dagger$. With
this identification $e(T^{\dagger\prime})_j=T^\dagger_j$ for each $j\in\mathbb N$.
 \end{thm}

\begin{proof} Since $X$ lies in $\mathscr A(\mathscr S)$, it has a  height filtration whose sections are direct sums of dual left cell modules $S_\omega$ (all having the same height).  For each $\omega\in\Omega$, $S_\omega$ appears as the lowest term in the filtration of the summand $X_\omega$
of $T^\dagger$. In particular, there is an inflation $S_\omega\to T^\dagger$.
Apply the remark made immediately above the statement of the theorem, with $Y=T^{\dagger\prime}
= T^\dagger\bigoplus X$.
There is a (nonzero) surjection (in $A^{\dagger\prime}$-mod)
$$A^{\dagger \prime}e=\Hom_H(T^\dagger, T^{\dagger\prime})\twoheadrightarrow\Hom_H(S_\omega,T^{\dagger\prime})$$
	using condition (2) above and the $\Ext^1_{\mathscr E(\mathscr S)}$-vanishing condition on 
	$X$ in the hypothesis of the theorem.  Observe $T^{\dagger\prime}=T^\dagger\oplus X\in{\mathscr A}(\mathscr S)$, since 
	$T^\dagger\in\mathscr A(\mathscr S)$ by (1) and $X\in\mathscr A(\mathscr S)$ by hypothesis. As an object in
	${\mathscr A}({\mathscr S})$,  $T^{\dagger\prime}$
	has a (height compatible) filtration with sections direct sums of modules $S_\omega$, $\omega\in\Omega$.
	Using the $\Ext^1_{\mathscr E(\mathscr S)}(-, T^{\dagger\prime})$-vanishing discussed above, we find that 
	the left regular module $_{A^{\dagger\prime}}A^{\dagger\prime}={_{A^{\dagger\prime}}\End}_H(T^{\dagger\prime})$ has a filtration with sections which are direct sum of  $A^{\dagger\prime}$-modules 
		$\Hom_H(S_\omega,T^{\dagger\prime})$. Thus, any
	irreducible $A^{\dagger\prime}$-module is a homomorphic image of some module
	 $\Hom_H(S_\omega, T^{\dagger\prime})$,
	and, thus,  by the surjection displayed above, of $A^{\dagger \prime}e$. 
	It follows 
	$A^{\dagger\prime}e$ is a projective generator for $A^{\dagger\prime}$-mod. 
The rest of the argument is either obvious, or follows from standard Morita theory, as in \cite[\S3.12]{Jac89}. 
For example, the functor $M\mapsto \Hom_{A^{\dagger\prime}}(A^{\dagger \prime}e,M)\cong eM$ from $A^{\dagger\prime}$-mod to
$\End_{A^{\dagger\prime}}(A^{\dagger\prime}e)^{\text{\rm op}}{\text{\rm -mod}}\cong eA^{\dagger\prime}e$-mod is an equivalence of
	categories.
\end{proof}

\begin{rem}  \label{rem4.2} The identification $e(T^{\dagger \prime})_j=T^\dagger_j$ above also yields an identification of the right $H$-modules
$e(T^{\dagger\prime}/(T^{\dagger\prime})_j)=T^\dagger/T^\dagger_j$. Thus, the given Morita equivalence takes the 
two-sided ideal
$$\Hom_H(T^{\dagger\prime}/(T^{\dagger^\prime})_j, T^{\dagger\prime})$$
 of $A^{\dagger\prime}$ to the two-sided ideal
$\Hom_H(T^\dagger/T^\dagger_j,T^\dagger)$ of $A^\dagger$. This will be very useful in our later discussion  of
stratified algebras and the ideals in their defining sequences.\end{rem}

In the next theorem, we work with the  commutative algebra $\sZ_{\mathfrak q}$
  the localization of $\sZ$ at a height one prime ideal $\mathfrak q=(\Phi_{2e})$, generated by a
cyclotomic polynomial $\Phi_{2e}$, $e\not=2$. (In \cite{DPS15}, the algebra $\sZ_{\mathfrak q}$ is denoted
$\sQ$.) Let $\widetilde H:= H_{\mathfrak q}$, the Hecke algebra over $\sZ_{\mathfrak q}$
with basis $T_w$, $w\in W$, and relations (\ref{relations}).  (In \cite{DPS15}, our $H$  here  is denoted 
$\mathcal H$, and $\widetilde H$ here is denoted $\widetilde{\mathcal H}$ there.)

 Recall that in \cite[Thm. 5.6]{DPS15}
the  $\sZ_{\mathfrak q}$-algebra $\widetilde{\mathcal A}^+$ is realized as an $\widetilde H$-endomorphism algebra 
\begin{equation}\label{endoalgebra}
\widetilde{\mathcal A}^+:=\End_{\widetilde H}(\wT^+),\quad{\text{\rm where}}\quad
 \wT^+:=\bigoplus_{\omega\in\Omega}\wT_\omega, \end{equation}
 where each $\wT_\omega $ is a right $\widetilde H$-module:
 If $\omega$ is a left cell containing the longest word $w_{\lambda,0}$ in a parabolic subgroup $W_\lambda$, then $\wT_\omega$
 in (\ref{endoalgebra})
 is a right $\widetilde H$ ``$q$-permutation'' module
  $x_\lambda\widetilde H$. (See \cite[(7.6.1)]{DDPW08}.) 
  (Actually, $\wT^+$ above allows for repetition of permutation modules as in (\ref{dagger}), whereas the
  original arguments in \cite{DPS15} do not. This does not affect the argument in \cite{DPS15}, and the endomorphism 
  algebra $\widetilde{\mathcal A}^+$ there
   is Morita equivalent to $\widetilde{\mathcal A}^+$ here in an obvious way.)
  Otherwise,
 the $\widetilde H$-modules $\wT_\omega$
are certain  right $\wH$-modules inductively constructed in \cite[\S5C; notation of \S5D]{DPS15} and called
$\widetilde X_\omega$, 
 indexed by the remaining left cells $\omega\in\Omega.$
  As pointed out in \cite[\S5C]{DPS15}, 
$\wT_\omega$ may be constructed as $(\mathcal T_\omega)_{\sQ}=(\mathcal T_\omega)_{\mathfrak q}$, where $\mathcal T_\omega$ 
is either a $q$-permutation modules $x_\lambda H$ or a $\mathcal Z$-free right $H$-module with certain
filtration properties.  Accordingly, we may write
\begin{equation}\label{endoalgebra2}
\mathcal A^+:=\End_H(\mathcal T^+), \quad{\text{\rm where}}\quad {\mathcal T}^+:=\bigoplus _{\omega\in\Omega}
\mathcal T_\omega\end{equation}
 with $\wT^+= \mathcal T^+_{\mathfrak q}:= (\mathcal T^+)_{\mathfrak q}$ and $\widetilde{\mathcal A}^+:=\mathcal A^+_{\mathfrak q}.$ Each of the localizations $(-)_{\mathfrak q}$ can be written as $(-)_{\sQ}$ in the
 notation of \cite{DPS15}. We also mention 
 that the expression ${\mathcal T}_\omega$ above is sometimes
 written as  $X_\omega$ in \cite{DPS15}.
   
    \medskip
  The proof of the following result requires a height function $\htt$ which is compatible with $\leq^\op_{LR}$, as used
  in \cite{DPS17a}. We quote below results from \cite[\S 5C,D]{DPS15}, which used a specific compatible height 
  function $f$ (called a sorting function in \cite{GGOR03}).
  However, this particular choice of height function is unnecessary.

  {\it In the following theorem, we assume that the parameters $c_s$ in (\ref{relations}) are all equal to 1.} However,
  we expect that the theorem holds for any choice of parameters that corresponds to a family of finite groups of
  Lie type. 

\begin{thm}\label{Steinberg}  Let $\mathcal A^+,\mathcal T^+, \mathfrak q$ be as above, and let $A^\dagger$ and $T^\dagger$
 be as in Theorem \ref{newlemma}. Then $\widetilde{\mathcal A}^+= {\mathcal A}^+_{\mathfrak q}$ is Morita equivalent to  $A^\dagger_{\mathfrak q}$ via an equivalence 
in which the height filtrations of $T^\dagger_{\mathfrak q}$ and $\widetilde{\mathcal T}^+=\mathcal T^+_{\mathfrak q}$
 all correspond. 
 \end{thm}

\begin{proof} It is easily seen from the construction in \cite[\S5C]{DPS15} that each $H$-module $\mathcal T_\omega$
and thus $\mathcal T^+$ belongs to the category ${\mathscr A}(\mathscr S)$ used in \cite[Thm. 4.7]{DPS17a}. According
to {\it op.cit.}, 
there is an inflation ${\mathcal T}^+\to X$ (in the exact category sense), where $X\in{\mathscr A}({\mathscr S})$
satisfies the vanishing 
conditions of Theorem \ref{newlemma}. By construction, the $H$-module
$X$ has a height filtration with sections $X_j/X_{j-1}$ which are direct sums of dual left cell modules $S_\omega$ of the same
height $j$.  Each such $S_\omega$ appears as a direct summand\footnote{In fact, using \cite[P4,P9]{Lus03}, which
certainly holds in the equal parameter setup of \cite{DPS15},  it can
be shown that $S_\omega$ is the full bottom section. The proof of \cite[Thm. 5.6]{DPS15} touches this point, but
without a proof. The fact here that $S_\omega$ a summand follows easily from the Kazhdan-Lusztig cell structure for
the (duals of) the various
parabolic right $H$-modules $x_\lambda H$ (for other $\mathcal T_\omega$ this property is automatic by construction).}
of the lowest term $\mathcal T_{\omega,i}$ in some summand $\mathcal T_\omega$ of $\mathcal T^+$ in the construction \cite[\S 5C,D]{DPS15}
of $\mathcal T^+$. Here
$i$ denotes a value of the height function $\htt$ on the lowest section for with $\mathcal T_{\omega,i}\not= 0$. 

 Let
$T^{\dagger}$ and $T^{\dagger\prime}=T^\dagger\oplus X$, with $X$ as above, be as in Theorem \ref{newlemma}. To prove the theorem it is enough, by
Theorem \ref{newlemma}, to prove Theorem 4.3 with $T^{\dagger\prime}$ replacing $T^\dagger$, 
and its $H$-endomorphism ring $A^{\dagger\prime}$ replacing
$A^\dagger$. 
As in the proof of Theorem \ref{newlemma},
we 
  find that
the module $_{A^{\dagger\prime}}A^{\dagger\prime}$ is filtered by modules $\Hom_{A^{\dagger\prime}}(S_\omega, T^{\dagger\prime})$, $\omega\in\Omega$. The vanishing of $\Ext^1_{{\mathscr E}({\mathscr S)}}(S_\nu, 
T^{\dagger\prime})$, for all $\nu\in\Omega$,  gives, by the remark immediately above Theorem \ref{newlemma},
a surjection
$$\Hom_H(\mathcal T^+,T^{\dagger\prime})\twoheadrightarrow
\Hom_H(S_\omega, T^{\dagger\prime}).$$
However, the exact sequence $0\to \mathcal T^+_{\mathfrak q}\to X_{\mathfrak q}\to (X/\mathcal T^+)_{\mathfrak q}\to 0,$ arising from the inflation above, is split
in mod-$H_{\mathfrak q}$ (note $H_{\mathfrak q}=\widetilde H$),  since
$X/\mathcal T^+\in \mathscr A({\mathscr S})$ and $\Ext^1_{\widetilde H}({\widetilde S_\omega},\widetilde{\mathcal T}^+)=0$
for all $\omega\in\Omega$ by \cite[Cor. 4.5(2), Prop. 5.5, \S 5C,D]{DPS15}. It follows that the
$A^{\dagger\prime}_{\mathfrak q}$-module $\Hom_{\widetilde H}(\wT^+,T^{\dagger\prime}_{\mathfrak q})$ is a projective
generator. Now argue  as in Theorem \ref{newlemma}. \end{proof}

% I think that the $X$ used in $T^{\dagger\prime}$ has the property 
%$$T^{\dagger\prime}_{\mathfrak p}=(T^{\dagger}+X)_{\mathfrak p}\cong T^{\dagger}_{\mathfrak p}\oplus 
%\mathcal T^+_{\mathfrak p}\oplus (X/\mathcal T^+)_{\mathfrak p},$$
%This becomes the same situation as in Lemma 4.1. So the argument above "However..." line seems not necessary??

\begin{rem} It was claimed without proof in \cite{DPS17a} that the algebra $A^\dagger$ constructed there had localizations
$A^\dagger_{\mathfrak q}$
agreeing, up to Morita equivalence, with the corresponding localizations in \cite{DPS15}, denoted ${\mathcal A}^+_{\mathfrak q}$ in our notation here, with $\mathfrak q$ as above. Theorem \ref{Steinberg} provides
the proof of this result.  This is potentially
important for future decomposition number calculations, since it was shown in \cite{DPS15} that
${\mathcal A}_{\mathfrak q}^+$ has a completion with module category equivalent to a Cherednik algebra module category
$\mathcal O$. 

We will give further applications of Theorem \ref{Steinberg} and \cite{DPS17a} in the final two sections.
 \end{rem}

\section{The Hecke algebras at good primes.} 
We begin this section in the setting of Theorem \ref{newlemma}.  In particular, we will use the algebra $A^\dagger=\End_H(T^\dagger)$ as in Theorem \ref{newlemma}. Using \cite[proof of Cor. 1.2.12]{DPS98a},  $A^\dagger$ is projective over $\sZ$. 
  
 By \cite[Thm. 4.9]{DPS17a} (which does not depend on the present discussion), $A^\dagger$ has
 a (strict) stratifying system, $\{\Delta(\omega):=\Hom_H(S_\omega,T^\dagger), P(\omega)\}_{\omega\in\Omega}.$

Also, using Remark \ref{rem4.2}, there are various ideals 
\begin{equation}
J_j:=\Hom_H(T^\dagger/T^\dagger_{N-j}, T^\dagger),\quad {\text{\rm where}}\quad N=\max\{\htt(\lambda)\}_{\lambda\in\Lambda}.\end{equation}
Obviously, we have a sequence $0=J_0\subseteq J_1\subseteq\cdots\subseteq J_{N}=A^\dagger$ of ideals.
%\ref{littlelemma}

\medskip
\begin{lem}\label{littlelemma} For each positive integer $j\leq N$, we have 

(1) the ideal $J_j$ is idempotent;

(2) the quotient $A^\dagger/J_j$ is a $\mathcal Z$-projective module;

(3) each section $J_j/J_{j-1}$ is projective as a left $A^\dagger/J_{j-1}$-module. \end{lem}

\begin{proof} By the proof of Theorem \ref{newlemma},  each  $J_j/J_{j-1}$ is isomorphic to $\Hom_H(T^\dagger_{N+1-{j}}/T^\dagger_{N-j}, T^\dagger)$ and thus to a direct sum of various
$\Hom_H(S_\omega,T^\dagger)$, with $\htt(\omega)=N+1-j$. We will use this throughout the proof.

\smallskip\noindent
{\bf Claim 1:} For each $j$ and $A^\dagger/J_{j}$-module $V$, we have $\Hom_{A^\dagger}(P(\omega),V)=0$
whenever $\htt(\omega)\geq N-j$.

\smallskip\noindent
{\it Proof of Claim 1:} This fact follows from the filtrations just mentioned and the axioms for a stratifying system when
$V$ is the left regular $A/J_{j}$-module. However, it reduces to this case through the projectivity of $P(\omega)$.
This proves Claim 1.
 
 \vspace{.3cm}
 
 We first prove (1). 
		As a consequence of Claim 1 and the filtrations above,  $\Hom_{A^\dagger}(J_j,V)=0$ whenever $V$ is an $ A^\dagger/J_j$-mod and
$\htt(\omega)\geq N-j$. This applies, in particular, to $V=J_j/J_j^2$, forcing the latter to be 0, i.e., $J_j^2=J_j$ and $J_j$ is idempotent. This
proves (1). 
%We do not see your argument.

 Next, (2) follows from the fact that $A^\dagger/J_j$ is filtered by various $\Delta(\omega)$, which are
$\mathcal Z$-projective.

 Finally, we prove (3). We prove $J_j/J_{j-1}$ is an $A^\dagger/J_{j-1}$-projective module. We know $J_j/J_{j-1}$ is filtered
 by $\Delta(\omega)=\Hom_H(S_\omega,T^\dagger)$, with $\htt(\omega)=N+1-i$. For each $\omega$ there is projective
 $A^\dagger$-module $P(\omega)$ with a $\Delta$-filtration. Also, $P(\omega)$ maps onto $\Delta(\omega)$. The other sections $\Delta(\tau)$ satisfy
  $\htt(\tau)>\htt(\omega)$. 
   
 \medskip\noindent
{\it Claim 2:} Suppose $\omega\in\Lambda$ has height $N-j$. Then
\begin{equation}\label{claim}P(\omega)/J_{j-1}P(\omega)\cong \Delta(\omega).\end{equation}

\medskip\noindent{\it Proof of Claim 2.}
 Equivalently, $J_{j-1}P(\omega)=K_\omega$, the kernel of the map $P(\omega)\twoheadrightarrow\Delta(\omega)$. Clearly, 
 $J_{j-1}P(\omega)\subseteq K_\omega$. 
 Also,  $P(\omega)/J_{j-1}P(\omega)$ is an $A^\dagger/J_{j-1}$-module, as is $K_\omega/J_{j-1}P(\omega)$. 
 However, Claim 1 implies that $\Hom_{A^\dagger}(\Delta(\omega), V)=0$ if $\htt(\omega)\geq N-j$ and $V$ is an
 $A^\dagger/J_j$-module. Since $K_\omega$ is filtered by such $\Delta$'s, it follows taking $V=K_\omega /J_{j-1}P_\omega=0$. This proves Claim 2.
 
 \medskip
 Next, note that $P(\omega)/J_{j-1}P(\omega)$ is a projective $A^\dagger/J_{j-1}$-module. We know that all
 direct summands of $J_j/J_{j-1}$ is a direct sum of projective $A/J_{j-1}$-modules (the various $\Delta(\omega)$
 with $\htt(\omega)=N+1-j$. Thus, $J_j/J_{j-1}$ is projective as an $A/J_{j-1}$-module. This completes the proof of (3) and thus the lemma. \end{proof}

We now define a specific multiplicative monoid $\mathbb S\subseteq \mathcal Z= {\mathbb Z}[t,t^{-1}]$ that will be
used in Theorem \ref{QHA}.  Namely, assume that $\mathbb S$ is generated by all the bad primes for the Weyl group $W$ together with the cyclotomic polynomial $\Phi_4(t)=t^2+1$. Let
$\mathcal Z^\natural=\mathbb S^{-1}\sZ$, the localization of $\sZ$ and $\mathbb S$. Put $H^\natural:=\mathbb S^{-1}H=\sZ^\natural\otimes_\sZ H$.  We consider the
right $H $-module $T^\dagger$ defined in \cite{DPS17a}, and put $T^{\dagger\natural}= {\mathbb S}^{-1}T^\dagger$.  Similarly, put
$A^{\dagger\natural}={\mathbb S}^{-1}A^\dagger\cong \End_{H^\natural}(T^{\dagger\natural})$.

Also, in Theorem \ref{QHA}, we presently require the equal parameter assumption of \cite{DPS15}. That is, we assume each
$c_s=1$ in (\ref{relations}). (The authors expect this  requirement will
be removed in a subsequent paper, and $t^2+1$ will be removed from $\mathbb S$. See Remark \ref{sub}.) In the proof, we shall also use Theorem \ref{Steinberg}; it serves here as a bridge between the results of \cite{DPS17a}, quoted above, and \cite{DPS15}. The module $\wT^+$ in Theorem 4.3 may be viewed as the base change to $\sQ$ of a module
$\mathcal T^{+\natural}$ for $H^\natural$. See \cite[\S5C]{DPS15} which even starts with modules for $H$. A similar remark
applies to $\widetilde{\mathcal A}^+$.

\begin{thm}\label{QHA} With the equal parameter assumption, $A^{\dagger\natural}$ is a split quasi-hereditary algebra over $\mathcal Z^\natural$. It has a split defining sequence obtained from that in Lemma \ref{littlelemma} by base-change to $\mathcal Z^\natural$.\end{thm}

\begin{proof}
Here we work over $\mathcal Z^\natural$ rather than the DVR $\sQ$ as 
in \cite{DPS15}. In fact, $\sQ$ is a localization of $\mathcal  Z^\natural$ at the ideal $\mathfrak q=(\Phi_{2e})$, $e\not=2$.  
  By \cite[Thm. 6.4]{DPS15},  $\mathcal A^{ +}_{\mathfrak q}$ ($=\widetilde{\mathcal A}^+$ in \cite{DPS15}) is quasi-hereditary, and thus ${\mathcal A}^{+\natural}(\mathfrak q)$ is quasi-hereditary. The proof there shows that 
  ${\mathcal A}^{+}_{\mathfrak q}$ is split quasi-hereditary, with the height function (used in Lemma \ref{littlelemma}) giving
  the required split defining sequence.  These properties pass to $A^{\dagger\natural}_{\mathfrak q}=A^\dagger_{\mathfrak q}$ by Theorem \ref{Steinberg}.
   Also, \cite[Thm. 4.2.2]{DPS98a} implies that $A^{\dagger\natural}(\mathfrak p)$ is split
  semisimple for all choices of height one primes $\mathfrak p= (p)$ with $p$  a good prime integer, or for $\mathfrak p=0$.  (Here we use the fact that $\mathbb S$ contains all the bad
  primes of $W$, appearing in the denominator of the generic degrees.) 
 
  The argument in \cite[op. cit.]{DPS98a} actually works in characteristic
 $0$, and shows that $A^{\dagger\natural}(\mathfrak p)$ is split semisimple, if $\mathfrak p= (f(t))$ for $f(t)\in\mathbb Z[t]$ a (nonscalar) primitive irreducible polynomial, and if $f(t)$ does not divide any product of cyclotomic polynomials $\Phi_e(t^2)$,
   $e\in\mathbb N^+$. We note that 
   \begin{equation}\label{cycl}
   \Phi_e(t^2)=\begin{cases}\Phi_{2e}(t), \quad e\,\,{\text{\rm even}}\\ \pm\Phi_{2e}(t)\Phi_{2e}(-t), \quad e\,\,
   {\text{\rm odd}}.\end{cases}\end{equation}
   by an elementary argument. Thus, if $f(t)$ divides any such product, it must be either $\Phi_{2e}(t)$
   or $\pm\Phi_{2e}(-t)$. The latter polynomial is conjugate, of course, to $\Phi_{2e}(t)$ by an automorphism of
   $\mathcal Z$. The associated prime ideals $\mathfrak p$ always give, if $e\not=2$,  split quasi-hereditary algebras  
   $A^{\dagger\natural}(\mathfrak p)$  with respect to the defining sequence defined using the height function, as discussed above. Now our assertion follows from Corollary \ref{OldCPSTheorem}.
     \end{proof}
 
 \begin{rem}\label{sub}  We expect to show in \cite{DPS17b} that Theorem \ref{QHA} holds when the multiplicative monoid $\mathbb S$
 is replaced by the smaller multiplicative monoid  $\mathbb S'$ generated by the bad primes for $W$. (In other words, the
 cyclotomic polynomial $t^2+1$ can be omitted. Also, in this more recent setting, we should not require equal parameters, but instead can use the
 setting of Theorem \ref{newlemma} in the current paper.)
  \end{rem}
 
In the proof of the following result, we fix a height function $\htt:\Lambda\to\mathbb N$ which is compatible with
the quasi-poset structure $\leq_{LR}^\op$ on $\Lambda$. We assume the stronger version of Theorem \ref{QHA} as 
described in Remark \ref{sub} above, and allow its relaxed hypothesis: thus, $\mathcal Z^\natural$ is obtained from
$\mathcal Z$ by inverting the bad primes of $W$, while $t^2+1$ plays no special role. For a version of Theorem \ref{Theorem4.2} proved using only Theorem \ref{QHA} as proved in the present paper, see Remark \ref{5rem} (a).

\begin{thm}\label{Theorem4.2} There is a family $\{E(\lambda)\}_{\lambda\in\Lambda}$ of $H^\natural$-modules
with the following properties: 

(1) Each $E(\lambda)$ is projective
as a $\nsZ$-module, and $E(\lambda)_K$ is the irreducible
$H_K$-module indexed by $\lambda$.  

(2) Let $\mathfrak m$ be any fixed 
maximal ideal of $\nsZ$. Let $D$ be an irreducible $H^\natural(\mathfrak
m)$-module. There is a unique $\lambda=\lambda(D)\in\Lambda$ such that
$[E(\lambda)(\mathfrak m): D]\not=0$ and such that $[E(\mu)(\mathfrak m):D]=0$ for any $\mu\in\Lambda$
with $\htt(\mu)\leq\htt(\lambda)$ and $\mu\not=\lambda$. 

(3)  Let $\mathfrak m$, $E(\lambda)$, and $\lambda=\lambda(D)$ be as in (2). Then $D$ is in the
head of $E(\lambda)(\mathfrak m)$.

(4) Using the notation of (3), we have $[E(\lambda)(\mathfrak m):D]=1$.
\end{thm}

\begin{proof} The notation $\Lambda$ was introduced in \S4 as a set indexing the irreducible left (or right) $H_K$-modules. The structure of $A^\dagger_K=A^{\dagger\natural}_K$ as endomorphism algebra shows that $\Lambda$ also indexes
in a corresponding way the irreducible left $A^\dagger_K=A^{\dagger\natural}_K$-modules. It is observed in \cite[Rem. 4.10]{DPS17a} that
the right regular module $H_H$ has a split embedding into $T^\dagger$. Using the idempotent projection
$e:T^\dagger\to H_H\subseteq T^\dagger$, we may identify 
$$H=eA^\dagger e\quad{\text{\rm and}}\,\, H^\natural=eA^{\dagger\natural}e.$$
Using either identification, multiplication by $e$ sends an irreducible $A^\dagger_K=A^{\dagger\natural}_K$-module
to an irreducible $H_K=H^\natural_K$-module. We may take multiplication by $e$ as defining the indexing correspondence
for $\Lambda$.

Multiplication by $e$ sends irreducible $A^{\dagger\natural}_k$-modules to irreducible or zero $H^\natural_k$-modules, whenever
$k$ is any field with a ring homomorphisms $\mathcal Z^\natural\to k$. Also, $e\,{\text{\rm head}}(V)\subseteq
{\text{\rm head}}(eV)$ for any left $A^{\dagger\natural}$-module $V$. 

Next, notice that Theorem \ref{Theorem4.2} makes sense if $H^\natural$ and $H_K$ are replaced in its wording by
$A^{\dagger\natural}$ and $A^{\dagger\natural}_K$, respectively. Moreover, this new version implies the original 
version using the above remarks, as the reader may check. We now prove the new version.

Recall that Theorem \ref{QHA} gives a defining sequence $0=J_0\subseteq J_1\subseteq\cdots
\subseteq J_n$ for $A^{\dagger\natural}$ associated to
the height function; see Lemma \ref{littlelemma}. The splitting property gives that each $\End_j:=\End_{A^{\dagger\natural}}(J_j/J_{j-1})$ is a direct product  of split Azumaya algebras. The latter are all  $\mathcal Z^\natural$-endomorphism algebras of projective $\mathcal Z^\natural$-modules. By \cite[p. 111, first paragraph]{Sw78}, the latter projective modules are free.
(We use here that $\mathcal Z^\natural$ is obtained from $\sZ=\mathbb Z[t,t^{-1}] $ by inverting some rational primes in $\mathbb Z$. This uses a Swan theorem version of Serre's conjecture. We remark that another way to insure projective modules over the base ring are free is to just pass to a localization at a maximal ideal, as in Remark \ref{5rem}(a).)

Thus, the direct product decomposition of $\End_j$ may be regarded as a direct sum of ideals associated to various
centrally primitive idempotents in $\End_j$. These idempotents may be obviously indexed as $e_\lambda$, $\lambda\in\Lambda_j$,
where $\Lambda_j$ corresponds to the irreducible left $(A^{\dagger\natural})_K$-modules in $(J_j/J_{j-1})_K$. Thus,
$J_j/J_{j-1}=\bigoplus_{\lambda\in \Lambda_j}e_\lambda(J_i/J_{i-1})$. The endomorphism algebra $\End_{A^{\dagger\natural}}(e_\lambda J_j/J_{j-1})$ is a full matrix algebra $M_{n_\lambda}(\mathcal Z^\natural)$ and the
$A^{\dagger\natural}_K$-module $(e_\lambda J_j/J_{j-1})_K$ is a direct sum $E_K(\lambda)^{\oplus _{n_\lambda}}$,
where $E_K(\lambda)$ is the irreducible $A^{\dagger\natural}_K$-module indexed by $\lambda$. 

Now fix any diagonal primitive idempotent $f_\lambda\in M_{n_\lambda}(\mathcal Z^\natural)$, and set
$$E(\lambda):= f_\lambda e_\lambda J_j/J_{j-1}.$$
Then $E(\lambda)$ is a projective $A^{\dagger\natural}/J_{j-1}$-module, and $E(\lambda)_K\cong E_K(\lambda)$.

We now check the properties in the new version of Theorem \ref{Theorem4.2}. Property (1) is obvious. Fix an irreducible
$A^{\dagger\natural}$-module $D$. Let $j$ be the smallest value with $[(J_j/J_{j-1}(\mathfrak m):D]\not=0$. 
Thus, $D$ is killed by $J_{j-1}$, so is an irreducible $A/J_{j-1}$-module. As such it appears in the head of
$A/J_{j-1}(\mathfrak m)$. Since it does not appear $(A/J_j)(\mathfrak m)$, it must appear in the head of 
$(J_j/J_{j-1})(\mathfrak m)$. In particular, it must appear in the head of $E(\lambda)(\mathfrak m)$ for some
$\lambda\in\Lambda_j$. Such a $\lambda$ is unique, since $E(\lambda)$ is a projective $A^{\dagger\natural}/J_{j-1}$-module, and $\Hom_{A^{\dagger\natural}}(E(\lambda),E(\mu))=0$ for $\lambda\not=\mu$ in $\Lambda_j$. This
proves (2) and (3).  Also, $\Hom_{A^{\dagger\natural}}(E(\lambda),E(\lambda))$ has rank one over $\mathcal Z^\natural$.
This proves (4) and the theorem is proved.
\end{proof}

\begin{rems} \label{5rem} (a) A local version of the theorem, using $\mathfrak m$ from the start, can be proved using the setting
of Theorem \ref{QHA} in this paper. Here one must assume that $t^2+1$ does not lie in $\mathfrak m$, and $\mathcal Z^\natural$ and $H^\natural$ should be replaced by ${\mathcal Z}^{\natural}_{\mathfrak m}$ and $H^\natural_{\mathfrak m}$,
respectively. There is only one choice, in this formulation for $\mathfrak m$ in (2), (3), (4). Still, it is
interesting that a result for the local case of $H^\natural_{\mathfrak m}$ can be deduced from properties of the various
$H^\natural_{\mathfrak p}$ for $\mathfrak p$ of height $\leq 1$ in $\mathcal Z$, which, in effect, happens in the
proof.

(b) For any fixed $\mathfrak m$, the formulation of Theorem \ref{Theorem4.2}
is (deliberately) similar to Geck's triangularization theorem \cite[Thm. 4.4.1]{GJ11}. The latter result uses a ``unipotent support'' function, an extension
of Lusztig's $a$-function, whereas we use a general height function. The context in \cite{GJ11} has the virtue of
relevance to unipotent conjugacy classes and character families, while ours has a similar relevance to
generalized $q$-Schur algebras.
 \end{rems}

\section{Bad primes and standardly stratified algebras.} We again use the notation of \S4, as introduced above Theorem 
\ref{newlemma},
and again let $A^\dagger$ denote the algebra $\End_{H}(T^\dagger)$  over $\sZ:=\mathbb Z[t,t^{-1}]$ (also discussed in \S4). Thus, Lemma
\ref{littlelemma} is available. Our aim in this section, is to give cruder versions of Theorems \ref{QHA} and \ref{Theorem4.2}
over $\mathcal Z$, rather
than $\sZ^\natural$, without any preference for ``good primes.'' Also, unequal parameters are allowed as in the
setting of \cite[Thm. 4.9]{DPS17a}.

For $\omega,\omega'\in\Omega$, define a pre-order
$$\omega\preceq\omega'\iff \htt(\omega)<\htt(\omega'), \text{{\rm or }} \htt(\omega)=\htt(\omega')\text{ {\rm and} }\omega\sim_{LR}\omega'.$$
(The same definition is in the preamble to \cite[Thm. 4.9]{DPS17a}.). Then $(\Omega,\preceq)$ becomes a quasi-poset and $\htt$ remains a height function with respect to $\preceq$.

The following result is stated and  proved in \cite[Thm. 4.9]{DPS17a}. It parallels Theorem \ref{QHA} above.

\begin{thm}\label{GreenThm4.9}  The $\sZ$-algebra $A^\dagger:=\End_{H}(T^\dagger)$ is standardly stratified. In fact,
 it has stratifying system, relative to the quasi-poset $(\Omega,\preceq)$, consisting of all $\Delta(\omega):=\Hom_{H}(S_\omega,T^\dagger)$, with $S_\omega$ ranging over the dual left cell modules.  \end{thm}

For $\omega\in\Omega$, let $[\omega]$ denote the two-sided cell to which $\omega$ belongs. It will be convenient in the theorem below to let $S[\omega]$ denote the two-sided cell module associated with $[\omega]$, viewed as a left $H$-module.\footnote{modulo
higher cells, using $\leq^{\op}_{LR}$.}  Part (1) of the theorem is implicit in the discussion above Theorem \ref{newlemma}.
The theorem parallels the first three parts of Theorem \ref{Theorem4.2}.

\begin{thm}\label{Theorem5.2} (1) All irreducible $H_K$-modules appear in some $S[\omega]_K$.

(2) Let $\mathfrak m$ be any fixed maximal ideal of $\sZ$. Let $D$ be an irreducible $H(\mathfrak m)$-module.
There is a unique two-sided cell $[\omega]= [\omega](D)$ with $[S[\omega](\mathfrak m):D]\not=0$ and $[S[\nu](\mathfrak m):D]=0$ for any
$\nu$ with $\htt(\nu)\leq\htt(\omega)$ and $[\nu]\not=[\omega]$. 

(3) For $[\omega]=[\omega](D)$ as above, $D$ is in the head of $S[\omega](\mathfrak m).$
\end{thm}

\begin{proof} It suffices to prove statements (2) and (3): We recall from \cite[Rem. 4.10]{DPS17a} that the
right $H$-module $H_H$ is a direct summand of $T^\dagger$. Let $e\in A^\dagger=\End_H(T^\dagger)$ be the projection $T^\dagger\to H_H\subseteq T^\dagger$.   Thus, $eA^\dagger e\cong H$. Also, any $eA^\dagger e$-module $e\Delta(\omega)$
identifies with the left $H$-module $S(\omega)$ using the isomorphisms 
$$\Hom_H(S_\omega,H_H)\cong \Hom_\sZ(S_\omega,\sZ)\cong S(\omega).$$
Here the left isomorphism follows from the bilinear form  induced identification
 $H_H  \cong \Hom_\sZ({_HH},\sZ)$ and the general isomorphism $\Hom_H(S_\omega,\Hom_{\sZ}({_H}H,\sZ))
 \cong \Hom_{\sZ}(S_\omega\otimes_HH,\sZ)$. 

Let $D$ be a given irreducible $H(\mathfrak m)$-module; it may also be viewed an irreducible left $H$-module. Viewing $H$ as
$eA^\dagger e$, there is a unique irreducible $A^\dagger$ module $L$ such
that $eL\cong D$. (This is easily argued by passing first to $A^\dagger(\mathfrak m)$, then to a simple algebra
 factor of its semisimple head.) Since $\{\Delta(\omega)\}_{\omega\in\Omega}$ is a stratifying
system,  there is a  $\omega\in\Omega$ with $\Hom_{A^\dagger}(\Delta(\omega),L)\not=0$. Also,
if $eL$ is a composition factor of $e\Delta(\nu)$ for some $\nu\in\Omega$, then $\omega\preceq\nu$, since
$\Hom_{A^\dagger}(P(\omega),\Delta(\nu))\not=0$. If the two-sided cells $[\omega]$ and $[\nu]$ are unequal,
 then $\htt(\nu)> \htt(\omega)$ by construction
of $\preceq$. This proves (2). Also, we get (3), since $e(\text{\rm head}(\Delta(\omega)))$
is contained in the head of $e\Delta(\omega).$ 
\end{proof}

\end{document}